\documentclass[numbib]{imaiai}
\usepackage{color}

\usepackage{graphicx}
\newcommand{\beas}{\begin{eqnarray*}}
\newcommand{\enas}{\end{eqnarray*}}
\newcommand{\bea}{\begin{eqnarray}}
\newcommand{\ena}{\end{eqnarray}}
\newcommand{\bms}{\begin{multline*}}
\newcommand{\ems}{\end{multline*}}
\newcommand{\bels}{\begin{align*}}
\newcommand{\enls}{\end{align*}}
\newcommand{\bel}{\begin{align}}
\newcommand{\enl}{\end{align}}
\newcommand{\qmq}[1]{\quad \mbox{#1} \quad}

\newcommand{\bbox}{\hfill $\Box$}
\newcommand{\ignore}[1]{}

\newcommand{\Bvert}{\left\vert\vphantom{\frac{1}{1}}\right.}

\def\blfootnote{\xdef\@thefnmark{}\@footnotetext}

\newcommand{\expect}[1]{E{\left[#1\right]}}

\newcommand{\dotp}[2]{\left\langle#1,#2\right\rangle}

\newcommand\argmin{\mathop{\mbox{argmin}}}

\begin{document}

\title{
Non-Gaussian Observations in Nonlinear Compressed Sensing via Stein Discrepancies
}
\shorttitle{Recovery from non-linear and non-Gaussian measurements}
\shortauthorlist{L. Goldstein, X. Wei}

\author{{
\sc Larry Goldstein}\thanks{Larry Goldstein was partially supported by NSA grant H98230-15-1-0250.}\\[2pt]
Department of Mathematics\\
University of Southern California\\
{larry@usc.edu}\\[6pt]
{\sc Xiaohan Wei}$^\dagger$\\[2pt]
Department of Electrical Engineering\\
University of Southern California\\
$^\dagger${\email{Corresponding author: xiaohanw@usc.edu}}
}

%
\maketitle
\\
\begin{abstract}
{Performance guarantees for estimates of unknowns in nonlinear compressed sensing models under non-Gaussian measurements can be achieved through the use of distributional characteristics that are sensitive to the distance to normality, and which in particular return the value of zero under Gaussian or linear sensing. The use of these characteristics, or discrepancies, improves some previous results in this area by relaxing conditions and tightening performance bounds. In addition, these characteristics are tractable to compute when Gaussian sensing is corrupted by either additive errors or mixing.}
{Signal reconstruction, semi-parametric single-index model, Stein coefficient, zero-bias coupling, generic chaining.}\\

2010 Math Subject Classification: 60D05, 94A12
\end{abstract}

\section{Introduction}

Consider the nonlinear sensing model where $(y_1,{\bf a}_1),\ldots,(y_m,{\bf a}_m)$ in $\mathbb{R} \times \mathbb{R}^d$ 
are i.i.d.\ copies of an observation and sensing vector pair $(y,{\bf a})$
satisfying 
\bea \label{eq:gen.lin.mod.1.8}
\expect{y|{\bf a}}=\theta(\langle {\bf a},{\bf x}\rangle),
\ena
where ${\bf a}$ is composed of entry-wise independent random variables distributed as $a$, a mean zero, variance one random variable.
Throughout the paper we assume that the function $\theta:\mathbb{R} \rightarrow \mathbb{R}$ is measurable, and ${\bf x} \in \mathbb{R}^d$ is an unknown, non-zero vector lying in a closed set $K \subseteq {\mathbb R}^d$. The goal is to recover $\mathbf{x}$ given the measurement pairs $\{(y_i,\mathbf{a}_i)\}_{i=1}^m$. We note that the magnitude of ${\bf x}$ is unidentifiable under the  model \eqref{eq:gen.lin.mod.1.8} as $\theta(\cdot)$ is unknown.
Hence in the following, by absorbing a factor of $\|{\bf x}\|$ into $\theta$, we may assume $\|{\bf x}\|_2=1$ without loss of generality. 

In \cite{ALPV14}, the authors consider model \eqref{eq:gen.lin.mod.1.8} under the one-bit sensing scenario where $y_1,\ldots,y_m$ lie in the two point set $\{-1,1\}$ and $\theta:\mathbb{R} \rightarrow [-1,1]$. They demonstrate that despite $\theta$ being unknown and potentially highly non-linear, performance guarantees can be provided for estimators $\widehat{\bf x}$ of ${\bf x}$ without additional knowledge of the structure of $\theta$, and in a way that allows for non-Gaussian sensing. 

Nonlinear compressed sensing beyond the one-bit model has also been considered in previous works under certain distribution assumptions. For example, \cite{plan2016generalized}  and \cite{goldstein2016structured} consider the nonlinear model \eqref{eq:gen.lin.mod.1.8} with measurement vectors $\mathbf{a}$ being Gaussian and an elliptical symmetric distribution respectively. More recently, \cite{pmlr-v70-yang17a} considers measurement vectors of general distribution via a score function method, under the assumption that the full knowledge of the distribution function is known. We also mention that the work \cite{erdogdu2016scaled} handles non-gaussian designs using the zero bias transform in order to study equivalences between Generalized and Ordinary least squares.

In \cite{ALPV14}, consideration of the non-Gaussian case introduces some challenges, reflected in potentially poor performance of the bounds, additional smoothness assumptions, and difficulties that may arise when the unknown is extremely sparse. We show many of these difficulties can be overcome through the introduction of various measures of the discrepancy between the sensing distribution of $a$ and the standard normal $g$. Though our main goal is to develop bounds that are sensitive to certain deviations from normality, and which in particular recover the previous results for Gaussian sensing and linear sensing as special cases, we also improve previous results by supplying explicit small constants in our recovery bounds.

Regarding notation, we generally adhere to the principle that random variables appear in upper case, but to be consistent with existing literature, and in particular with \cite{ALPV14}, we make an exception for the components of the sensing vector, generically denoted by $a$ and the Gaussian by $g$, and also for the observed values, denoted by $y$. Vectors are in bold face.

\subsection{Estimator and main result}

Given the pairs $\{(y_i,{\bf a}_i)\}_{i=1}^m$ generated by the model \eqref{eq:gen.lin.mod.1.8}, let
\bea\label{empirical-loss}
L_m(\mathbf{t}):=\|\mathbf{t}\|_2^2-\frac{2}{m}\sum_{i=1}^my_i\dotp{\mathbf{a}_i}{\mathbf{t}} \qmq{for ${\bf t} \in K$,}
\ena
which is an unbiased estimator of
\bea\label{expected-loss}
L(\mathbf{t}):=\|\mathbf{t}\|_2^2-2\expect{y\dotp{\mathbf{a}}{\mathbf{t}}}.
\ena
As the components of ${\bf a}$ have mean zero, variance one and are independent,  $\expect{\mathbf{a}\mathbf{a}^T}=\mathbf{I}_{d\times d}$, and therefore minimizing $L(\mathbf{t})$ 
is equivalent to minimizing the quadratic loss $\expect{\left(y-\dotp{\mathbf{a}}{\mathbf{t}}\right)^2}$. Thus, we define the estimator 
\bea\label{estimator}
\widehat{\mathbf{x}}_m:=\argmin_{\mathbf{t}\in K}L_m(\mathbf{t}).
\ena
For simplicity of notation, we will write
\begin{equation}\label{def-empirical-function}
f_{\mathbf{x}}(\mathbf{t}):=\frac{1}{m}\sum_{i=1}^m y_i\dotp{\mathbf{a}_i}{\mathbf{t}}.
\end{equation}
To state the main result, we need the following three definitions:
\begin{definition}[Gaussian mean width]\label{def:gmw}
For $\mathbf{g}\sim\mathcal{N}(0,\mathbf{I}_{d\times d})$, the Gaussian mean width of a set $\mathcal T\subseteq\mathbb{R}^d$ is 
\[\omega(\mathcal{T})=\expect{\sup_{{\bf t}\in \mathcal T}\dotp{\mathbf{g}}{{\bf t}}}.\]
\end{definition}

\begin{remark}
In \cite{ALPV14}, the definition of Gaussian mean width of a set $\mathcal T$ is taken to be
\[\omega(\mathcal T)=\expect{\sup_{{\bf t}\in\mathcal T - \mathcal T}\dotp{\mathbf{g}}{{\bf t}}},\]
where the supremum is over the 
Minkowski difference.
Here, for ease of presentation, we adopt the somewhat more ``classical'' Definition \ref{def:gmw} that appears in earlier works in the literature, such as \cite{RV08}. These two definitions are equivalent up to a constant as $\expect{\sup_{{\bf t}\in\mathcal T - \mathcal T}\dotp{\mathbf{g}}{{\bf t}}}= 2\expect{\sup_{{\bf t}\in\mathcal T}\dotp{\mathbf{g}}{{\bf t}}}$,
which can be seen using the symmetry of the distribution of $\mathbf{g}$.
\end{remark}

\begin{remark}[Measurability issue]\label{remark-measurability}
The precise meaning of~~$\expect{\sup_{t\in \mathcal{T}}X(t)}$ for an arbitrary process $\{X(t)\}_{t\in\mathcal{T}}$ is not clear if $\mathcal T$ is uncountable.
In fact, for an uncountable index set $\mathcal T$, the function $\sup_{t\in {\mathcal T}}X(t)$ might not be measurable. Letting $(\Omega,\mathcal{E},P)$ be the underlying probability space, well known counter examples exist even in the case where $X(\cdot)$ is jointly measurable on the product space $(\Omega\times \mathcal T,~\mathcal{E}\otimes\Psi)$ (first constructed by Luzin and Suslin), where 
$\Psi$ is a Borel $\sigma$-algebra on $\mathcal T$. However, when $\mathcal T$ is a Borel measurable subset of $\mathbb{R}^d$ (which is the case we are interested in) and $X(\cdot)$ is jointly measurable on $(\Omega\times \mathcal T,~\mathcal{E}\otimes\Psi)$, one can show that the $\sup_{t\in \mathcal T}X(t)$ is always measurable. 

Indeed, 
$\sup_{t\in \mathcal T}X(t)$ is measurable if and only if the set $\{\sup_{t\in \mathcal T}X(t)> c\}\in\mathcal{E}$ for any $c\in\mathbb{R}$. On the other hand, $\{\sup_{t\in\mathcal T}X(t)> c\}=P_{\Omega}\{X(\cdot)> c\}$, where 
for any set $A\in\Omega\times\mathcal T$, 
$P_{\Omega}A:=\{\omega\in\Omega:(\omega,t)\in A\}$ is the projection of the set $A$ onto $\Omega$. Then, the measurability comes from the following theorem in \cite{Co80}: If $(\Omega, \mathcal{E})$ is a measurable space and $\mathcal T$ is a Polish space, then, the projection onto $\Omega$ of any product measurable subset of $\Omega\times\mathcal T$ is also measurable. 
\end{remark}

\begin{definition}[$\psi_q$-norm]\label{def:psi_qnorm}
The $\psi_q$-norm of a real valued random variable $X$ is given by
\[\|X\|_{\psi_q}=\sup_{p\geq1}p^{-\frac{1}{q}}(\expect{|X|^p})^{\frac1p}.\]
In particular, for $q=1$ and $q=2$ respectively, the value of $\psi_q$ is called the subexponential and subgaussian norm, and we say $X$ is subexponential or subgaussian when
 $\|X\|_{\psi_1}<\infty$ or $\|X\|_{\psi_2}<\infty$.
\end{definition}

The subgaussian $q=2$ case of Definition \ref{def:psi_qnorm} is the most important. Though here the $\psi_2$-norm we have chosen to use is based on comparing the growth of a distribution's absolute moments to that of a normal, definitions equivalent up to universal constants can also be stated in terms of comparisons of tail decay or of the Laplace transform of $X$, among others.

\begin{remark}\label{norm-justify}
It is easily justified that $\|\cdot\|_{\psi_q}$ for $q \ge 1$ defines a norm with identification of almost everywhere equal random variables. Here we only check the triangle inequality as it is immediate that $\|\cdot\|_{\psi_q}$ is homogeneous and separates points. Indeed, for any two random variables $X$ and $Y$, the Minkowski inequality yields that
\begin{align*}
\|X+Y\|_{\psi_q}=\sup_{p\geq1}p^{-\frac{1}{q}}(\expect{|X+Y|^p})^{\frac1p}
\leq\sup_{p\geq1}p^{-\frac{1}{q}}\left((\expect{|X|^p})^{\frac1p}+(\expect{|Y|^p})^{\frac1p}\right)
\leq\|X\|_{\psi_q}+\|Y\|_{\psi_q}.
\end{align*}
\end{remark}

\begin{definition}[Descent cone]
The descent cone of a set $\mathcal T\subseteq\mathbb{R}^d$ at any point ${\bf t}_0\in \mathcal T$ is defined as
\[D(\mathcal T,{\bf t}_0)=\{\tau\mathbf{h}:~\tau\geq0, \mathbf{h}\in \mathcal T-{\bf t}_0\}.\]
\end{definition}

\begin{theorem}\label{main-theorem}
Let $\mathbf{a}=(a_1,\ldots,a_d)$ where $a_1,\ldots,a_d$ are i.i.d.\ copies of a random variable $a$ with a centered
 subgaussian distribution having unit variance, and let $\{(y_i,\mathbf{a}_i)\}_{i=1}^m$ be i.i.d.\ copies of the pair $(y,\mathbf{a})$ where $y$, given by the 
sensing model \eqref{eq:gen.lin.mod.1.8}, is assumed to be subgaussian. If $K$ is a closed, measurable subset of $\mathbb{R}^d$
and 
$\lambda {\bf x} \in K$ where
\bea \label{def:lambda}
\lambda=\expect{y\dotp{\mathbf{a}}{\mathbf{x}}},
\ena
then for all $u\geq2$, with probability at least $1-4e^{-u}$, the estimator $\widehat{\mathbf{x}}_m$ given by \eqref{estimator} satisfies 
\[\left\|\widehat{\mathbf{x}}_m-\lambda\mathbf{x}\right\|_2
\leq2\alpha+C_0(\|a\|_{\psi_2}^2+\|y\|_{\psi_2}^2)\frac{\omega(D(K,\lambda\mathbf{x})\cap\mathbb{S}^{d-1})+u}{\sqrt{m}},\]
for all $m\geq \omega(D(K,\mathbf{x})\cap\mathbb{S}^{d-1})^2$ and some constant $C_0>0$, where
\begin{equation}\label{def:alpha}
\alpha=\sup\{  \left|\expect{  y \dotp{\mathbf{a}}{\mathbf{t}}  }-\lambda\dotp{\mathbf{x}}{\mathbf{t}} \right|, \mathbf{t} \in B_2^d \},
\end{equation}
and $\mathbb{S}^{d-1}$ and $B_2^d$ are the unit Euclidean sphere and ball in $\mathbb{R}^d$, respectively.
\end{theorem}

We note that $\alpha=0$ under the conditions of Theorems \ref{thm:alpha.via.sc} and \ref{thm:sign.zero.bias}, and also when $\theta$ is linear. Hence Theorem \ref{main-theorem} recovers
results for the normal and linear compressed sensing models as special cases.

\begin{remark}
At first glance it may seem surprising that the least squares type estimator \eqref{estimator}, which is well known to work when $\theta$ is linear, succeeds in such greater generality. The appearance of the factors $\lambda$ and $\alpha$ in \eqref{def:lambda} and \eqref{def:alpha}, respectively, may also be non-intuitive. The following explanations may shed some light. 

First, regarding the scaling factor $\lambda$, one can easily verify that if $\theta(w)=\mu w$, a linear function, then $\lambda=\mu$. Hence, in this case $\theta(\langle {\bf a}, {\bf x}\rangle )=\lambda \langle {\bf a}, {\bf x}\rangle = \langle {\bf a}, \lambda {\bf x}\rangle$, which behaves as though the unknown vector to be estimated has length $\lambda$, possibly different from one, the assumed length of ${\bf x}$.

Next, we present Lemma \ref{lem:Lt.minus.Llx}, used later in the proof of Theorem \ref{main-theorem}, to give some intuition as to why the proposed estimator succeeds when $\theta$ is non-linear. Let $L$ be as in \eqref{expected-loss}, the expectation of the function $L_m$ whose argument at the minimum defines the estimator $\widehat{\mathbf{x}}_m$. 

\begin{lemma} \label{lem:Lt.minus.Llx}
	For any $\mathbf{t}\in K$, we have
	\[L(\mathbf{t})-L(\lambda\mathbf{x})\geq\|\mathbf{t}-\lambda\mathbf{x}\|_2^2-2\alpha\|\mathbf{t}-\lambda\mathbf{x}\|_2,\]
	where $\lambda$ and $\alpha$ are defined in \eqref{def:lambda} and \eqref{def:alpha}, respectively.
\end{lemma}
\begin{proof}
	For any $\mathbf{t}\in K$, 
	\begin{multline*}
	L(\mathbf{t})-L(\lambda\mathbf{x})= \|\mathbf{t}\|_2^2-\|\lambda\mathbf{x}\|_2^2
	-2\expect{y\dotp{\mathbf{a}}{\mathbf{t}-\lambda\mathbf{x}}} 
	\\
	\geq\|\mathbf{t}\|_2^2-\|\lambda\mathbf{x}\|_2^2-2\lambda\dotp{\mathbf{t}-\lambda\mathbf{x}}{\mathbf{x}}-2\alpha\|\mathbf{t}-\lambda\mathbf{x}\|_2\\
	=\|\mathbf{t}-\lambda\mathbf{x}\|_2^2-2\alpha\|\mathbf{t}-\lambda\mathbf{x}\|_2,
	\end{multline*}
	where the inequality follows from \eqref{def:alpha}.
\end{proof}

Hence, if one could minimize $L$ instead of $L_m$ (the difference in practice being controlled by a generic chaining argument), when $\lambda {\bf x} \in K$, the set over which $L$ is minimized, one obtains
\bea \label{eq:minL}
\|\widehat{\bf x}_m-\lambda {\bf x}\|_2^2 \le [L(\widehat{\bf x}_m) - L(\lambda {\bf x})] + 2 \alpha  \|\widehat{\bf x}_m-\lambda {\bf x}\|_2 \le 2 \alpha  \|\widehat{\bf x}_m-\lambda {\bf x}\|_2,
\ena
and therefore that
\beas
\|\widehat{\bf x}_m-\lambda {\bf x}\|_2 \le 2 \alpha.
\enas
\end{remark}
From the inequality in the proof of Lemma \ref{lem:Lt.minus.Llx} one can see that $\alpha$ is the `price' for replacing the non-linearity inherent in $y$ with a simpler inner product, as supported by the fact that $\alpha=0$ when $\theta$ is linear.  In addition, parts (a) and (b) of Theorem \ref{thm:alpha.via.sc} to follow show that $\alpha$ is again zero when $\theta$ is Lipschitz, or has bounded second derivative, and the sensing vector is composed of independent Gaussian variables. Theorem \ref{thm:sign.zero.bias} provides this same conclusion when $\theta$ is the sign function. Hence, in these cases, minimizing $L$ would lead to exact recovery. 

As mentioned earlier, the length of the unknown vector ${\bf x}$ in \eqref{eq:gen.lin.mod.1.8} is not identifiable due to the generality in $\theta$ that the model allows. However, if one has prior knowledge that $\|\mathbf{x}\|_2 = 1$, the following corollary to Theorem \ref{main-theorem} shows that rescaling $\widehat{\mathbf{x}}_m$ to have norm 1 gives an estimator of the true vector $\mathbf{x}$. The idea underlying the corollary was originally developed in \cite{goldstein2016structured}.

\begin{corollary}\label{cor:normalize}
Let the conditions of Theorem \ref{main-theorem} be in force, and suppose that $\|\mathbf{x}\|_2 = 1$ and $\lambda > 0$. Define the normalized estimator $\overline{\mathbf{x}}_m$, as
\[
\overline{\mathbf{x}}_m := 
\begin{cases}
\widehat{\mathbf{x}}_m/\|\widehat{\mathbf{x}}_m\|_2,~~&\text{if}~\widehat{\mathbf{x}}_m\neq0, \\
0,~~&\text{if}~\widehat{\mathbf{x}}_m=0.
\end{cases}
\]
Then there exists a constant $C_0>0$ such that for all $u \ge 2$, with probability at least $1-4e^{-u}$, 
\[\left\| \overline{\mathbf{x}}_m - \mathbf{x}\right\|_2
\leq\frac{4\alpha}{\lambda}+2C_0(\|a\|_{\psi_2}^2+\|y\|_{\psi_2}^2)\frac{\omega(D(K,\lambda\mathbf{x})\cap\mathbb{S}^{d-1})+u}{\lambda\sqrt{m}},\]
whenever $m\geq \omega(D(K,\mathbf{x})\cap\mathbb{S}^{d-1})^2$. 
\end{corollary}
\begin{figure}[htbp]
 \centering
   \includegraphics[width=4in]{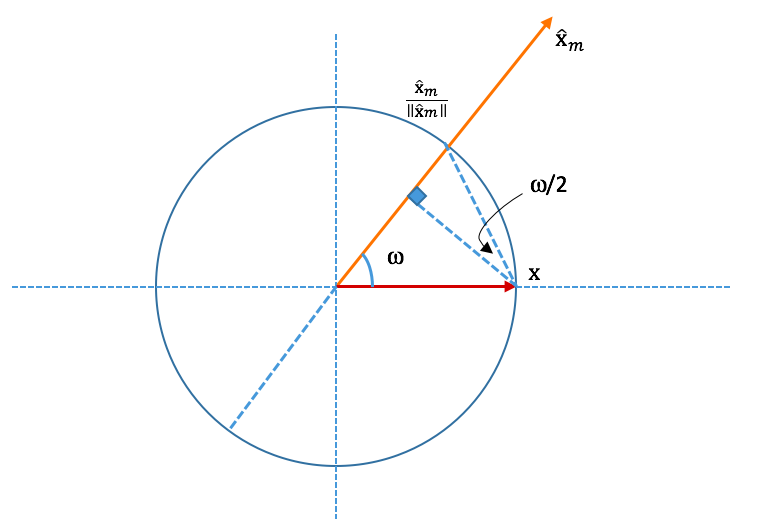} 
   \caption{Illustration of the geometric relation between the estimator $\widehat{\mathbf{x}}_m$ and the true vector $\mathbf{x}$.}
   \label{fig:Stupendous1}
\end{figure}

\begin{proof}
By Theorem \ref{main-theorem}, we know that with probability at least $1-4e^{-u}$
\[
\left\|\widehat{\mathbf{x}}_m-\lambda\mathbf{x}\right\|_2\leq B \qmq{where} B = 2\alpha+C_0(\|a\|_{\psi_2}^2+\|y\|_{\psi_2}^2)\frac{\omega(D(K,\lambda\mathbf{x})\cap\mathbb{S}^{d-1})+u}{\sqrt{m}}.
\]
Since $\lambda >0$, it follows that on this event
\begin{equation*}
\left\|\frac{\widehat{\mathbf{x}}_m}{\lambda} - \mathbf{x}\right\|_2\leq \frac{B}{\lambda}.
\end{equation*}
Let $\omega \in [0,\pi)$ be the angle between $\widehat{\mathbf{x}}_m$ and $\mathbf{x}$ (See Figure \ref{fig:Stupendous1}). First consider the case where either $\omega\geq\frac{\pi}{2}$ or $\widehat{\mathbf{x}}_m =0$.
Then $\langle\widehat{\mathbf{x}}_m,\mathbf{x}\rangle \leq 0$, and we have from the above inequality,
$$\frac{B}{\lambda}\geq\left\|\frac{\widehat{\mathbf{x}}_m}{\lambda} - \mathbf{x}\right\|_2 = \sqrt{\|\widehat{\mathbf{x}}_m/\lambda\|_2^2 - 2 \langle\widehat{\mathbf{x}}_m,\mathbf{x}\rangle/\lambda + \| \mathbf{x}\|_2^2}\geq\|\mathbf{x}\|_2=1.$$
Hence, applying the triangle inequality, we have
\begin{equation*}
\left\|\overline{\mathbf{x}}_m - \mathbf{x}\right\|_2\leq 2 \leq \frac{2B}{\lambda}.
\end{equation*}

In the remaining case where $\omega<\frac{\pi}{2}$ and $\widehat{\mathbf{x}}_m \neq0$,
as can be seen with the help of Figure \ref{fig:Stupendous1},
\begin{multline*}
\left\|\frac{\widehat{\mathbf{x}}_m}{\|\widehat{\mathbf{x}}_m\|_2} - \mathbf{x}\right\|_2
= \frac{\text{dist}\left(\mathbf{x},\text{span}\left(\widehat{\mathbf{x}}_m\right)\right)}{\cos(\omega/2)}
\leq \frac{\text{dist}\left(\mathbf{x},\text{span}\left(\widehat{\mathbf{x}}_m\right)\right)}{\cos(\pi/4)}
\leq \frac{\left\|\widehat{\mathbf{x}}_m/\lambda - \mathbf{x}\right\|_2}{\cos(\pi/4)}
\leq \frac{\sqrt{2}B}{\lambda},
\end{multline*}
where $\text{dist}\left(\mathbf{x},\text{span}\left(\widehat{\mathbf{x}}_m\right)\right)$ denotes the distance of the vector $\mathbf{x}$ to the linear span of 
$\widehat{\mathbf{x}}_m$, the first inequality follows from $\omega<\frac{\pi}{2}$ and the second inequality follows from the fact that $\frac{\widehat{\mathbf{x}}_m}{\lambda}$ is in the linear span of $\widehat{\mathbf{x}}_m$. 
Combining the above two cases completes the proof. 
\end{proof}

\begin{remark}
We compare the result in  Corollary \ref{cor:normalize} with 
Lemma 2.2 of \cite{ALPV14}, where a nearly identical bound is presented under the additional assumptions that 
$\{y_i\}_{i=1}^m$ take values in $\{-1,1\}$, $\theta:\mathbb{R}\rightarrow[-1,1]$, that $K$ lies in a unit Euclidean ball $B_2^d$, and $g \sim \mathcal{N}(0,1)$.
Specifically, under the preceding assumptions it is shown that
\[\left\|\widehat{\mathbf{x}}_m-\mathbf{x}\right\|_2\leq\frac{4\alpha}{\lambda}+C\|a\|_{\psi_2}\frac{\omega(K)+u}{\lambda\sqrt{m}},\]
with probability at least $1-4e^{-u^2}$.   Under the normality assumption $\dotp{\mathbf{g}}{\mathbf{x}} \sim {\cal N}(0,1)$ and $\lambda$ of \eqref{def:lambda} specializes to $E[g\theta(g)]$. Here, we are able to obtain a more general result that allows $y$ to be sub-gaussian rather than restricting it to lie in $\{-1,1\}$, which comes at the extra cost of a term that is of the same order as previously existing ones in the bound, and in particular which vanish as $m\rightarrow\infty$. Lastly, allowing $y$ to be sub-gaussian, the variable $y\dotp{\mathbf{a}}{\mathbf{t}}$ is sub-exponential for all $\mathbf{t}\in\mathbb{R}^d$, as opposed to being sub-gaussian as in \cite{ALPV14}. This additional generality necessitates a generic chaining argument to obtain
the sub-exponential concentration bound.
\end{remark}

This paper is organized as follows. In Section \ref{sec:disc.bounds} we introduce two measures of a distribution's discrepancy from the normal that have their roots in Stein's method, see \cite{CGS10}, \cite{St72}. The zero bias distribution is introduced first, being relevant for both Sections \ref{sec:lip} and \ref{sec:sign}, that consider the cases where $\theta$ is a smooth function, and the sign function, respectively. Section \ref{sec:lip} further introduces a discrepancy measure based on Stein coefficients, and Theorem \ref{thm:alpha.via.sc} provides bounds on $\alpha$ of \eqref{def:alpha} in terms of these two measures, when $\theta$ is Lipschitz and when it has a bounded second derivative. Section \ref{sec:lip} also defines two specific error models on the Gaussian, an additive one in \eqref{eq:geps.additive}, and the other via mixtures, in \eqref{eq:geps.mixture}. Theorem \ref{thm:Lip.Error.Models} shows the behavior of the bound on $\alpha$ in these two models as a function of the amount $\epsilon \in [0,1]$ the Gaussian is corrupted, tending to zero as $\epsilon$ becomes small.

Section \ref{sec:sign} provides corresponding results when $\theta$ is the sign function, specifically in Theorems \ref{thm:sign.zero.bias} and \ref{thm:Sign.Error.Models}.
Section \ref{sec:SC.ZB.relation} studies some relationships between the two discrepancy measures applied, and also to the total variation distance. Theorem \ref{main-theorem} is proved in Section \ref{sec:proof.of.main.thm}. The presentation of the postponed proofs of some results used earlier appear in an Appendix in Sections \ref{app:A} and \ref{app:additional}.

\section{Discrepancy bounds via Stein's method}\label{sec:disc.bounds}
Here we introduce two measures of the sensing distribution's proximity to normality that can be used to bound $\alpha$ in \eqref{def:alpha}. 
In Sections \ref{sec:lip} and \ref{sec:sign} we consider the cases where $\theta$ is a Lipschitz function, and the sign function, respectively; the difference in the degree of smoothness in these two cases necessitates the use of different ways of measuring the discrepancy to normality. 

An observation that will be useful in both settings is that by definition \eqref{def-empirical-function}, for any $\mathbf{t}\in\mathbb{R}^d$, we have
\begin{multline} \label{eq:Efxx.to.expand}
E[f_{{\bf x}}({\bf t})]=E[y\langle {\bf a}, {\bf t}\rangle]=E[E[y\langle {\bf a}, {\bf t}\rangle |{\bf a}]]
=E[\langle {\bf a}, {\bf t}\rangle \theta(\langle {\bf a}, {\bf x}\rangle)]=\langle {\bf v}_{\bf x},{\bf t}\rangle \\ \qmq{where} {\bf v}_{\bf x}=E[{\bf a}\theta(\langle {\bf a}, {\bf x} \rangle)].
\end{multline}
Specializing to the case where ${\bf t}={\bf x}$, we may therefore express $\lambda$ in \eqref{def:lambda} as
\bea \label{def:altlambda}
\lambda=\langle {\bf v}_{\bf x},{\bf x}\rangle = \expect{\langle {\bf a},{\bf x}\rangle \theta(\langle {\bf a},{\bf x}\rangle)}.
\ena

In the settings of both Sections \ref{sec:lip} and \ref{sec:sign}, we require facts regarding the zero bias distribution, and depend on \cite{GR97} or \cite{CGS10} for properties stated below. With ${\cal L}(\cdot)$ denoting distribution, or law, given a mean zero distribution ${\cal L}(a)$ with
finite, non-zero variance $\sigma^2$,  there exists a unique law ${\cal L}(a^*)$, termed the `$a$-zero bias' distribution, characterized by the satisfaction of
\bea \label{eq:zb.char}
E[af(a)]=\sigma^2 E[f'(a^*)] \qmq{for all Lipschitz functions $f$.}
\ena
The existence of the variance of $a$, and hence also its second moment, guarantees that the expectation on the left, and hence also on the right, exists.

Letting
\beas 
{\rm Lip}_1=\{g:\mathbb{R} \rightarrow \mathbb{R} \qmq{satisfying} |g(y)-g(x)| \le |y-x|\},
\enas
we recall that the Wasserstein, or $L^1$ distance between the laws ${\cal L}(X)$ and ${\cal L}(Y)$ of two random variables $X$ and $Y$ can be defined as
\beas
d_1({\cal L}(X),{\cal L}(Y)) = \sup_{f \in {\rm Lip}_1} |Ef(X)-Ef(Y)|,
\enas
or alternatively as
\bea \label{eq:d1.inf}
d_1({\cal L}(X),{\cal L}(Y)) = \inf_{(X,Y)}E|X-Y|
\ena
where the infimum is over all couplings $(X,Y)$ of random variables having the given marginals. The infimum is achievable for real valued random variables, see \cite{Ra91}.

Now we define our first discrepancy measure by 
\bea \label{def:gammaa}
\gamma_{{\cal L}(a)} = d_1(a,a^*).
\ena
Stein's characterization \cite{St72} of the normal yields that ${\cal L}(a^*)={\cal L}(a)$ if and only if $a$ is a mean zero normal variable. Further, with some abuse of notation, writing $\gamma_a$ for \eqref{def:gammaa} for simplicity, 
Lemma 1.1 of \cite{G10} yields that if $a$ has mean zero, variance $1$ and finite third moment, then
\bea \label{eq:Edifasax3}
\gamma_a \le \frac{1}{2}E|a|^3,
\ena
so in particular $\gamma_a < \infty$ whenever $a$ has a finite third moment.  
In the case where $Y_1,\ldots,Y_n$ are independent mean zero random variables with finite, non-zero variances $\sigma_1^2,\ldots,\sigma_n^2$, having sum $Y=\sum_{i=1}^n Y_i$ with variance $\sigma^2$, we may construct $Y^*$ with the $Y$-zero biased distribution by letting
\bea \label{eq:Y*.replace.one}
Y^*=Y-Y_I+Y_I^* \qmq{where} P[I=i]=\frac{\sigma_i^2}{\sigma^2},
\ena
where $Y_i^*$ has the $Y_i$-zero biased distribution and is independent of $Y_j, j \not =i$, and where the random index $I$ is independent of $\{Y_i,Y_i^*, i=1,\ldots,n\}$. We will also make use of the fact that for any $c \not =0$
	\bea \label{eq:zero.bias.homogeneous}
	{\cal L}((ca)^*)={\cal L}(ca^*).
	\ena

\subsection{Lipschitz functions} \label{sec:lip}
When $\theta$ is a Lipschitz function inequality \eqref{eq:alpha.le.expect1-T} of Theorem \ref{thm:alpha.via.sc} below gives a bound on $\alpha$ in \eqref{def:alpha} in terms of Stein coefficients. We say $T$ is a Stein coefficient, or Stein kernel, for a random variable $X$ with finite, non zero variance when 
\bea \label{eq:Stein.Coeff}
E[Xf(X)]=E[Tf'(X)]
\ena
for all Lipschitz functions $f$. Specializing \eqref{eq:Stein.Coeff} to the cases where $f(x)=1$ and $f(x)=x$ we find
\bea \label{eq:Eazero.Ehsigma2}
E[X]=0 \qmq{and} {\rm Var}(X)= E[T].
\ena

By Stein's characterization \cite{St72}, the distribution of $X$ is normal with mean zero and variance $\sigma^2$ if and only if $T=\sigma^2$. Correspondingly, for unit variance random variables we will define our second discrepancy measure as $E|1-T|$. If $c$ is a non-zero constant and $T_X$ is a Stein coefficient for $X$, then $c^2T_X$ is a Stein coefficient for $Y=cX$. Indeed,  
with $h(x)=f(cx)$ below we obtain changed $g$ to $h$ to avoid confusion with normal
\begin{multline} \label{eq:sc.linear}
E[Yf(Y)]=cE[Xf(cX)]=cE[Xh(X)]=cE[T_Xh'(X)]\\=cE[cT_Xf'(cX)] =E[c^2T_Xf'(Y)].
\end{multline} 
Stein coefficients first appeared in the work of \cite{CP92}, and were further developed in \cite{Ch09} for random variables that are functions of Gaussians; we revisit this later point in Section \ref{sec:SC.ZB.relation}.


The following result considers two separate sets of hypotheses on the unknown function $\theta$ and the sensing distribution $a$. The assumptions leading to the bound \eqref{eq:alpha.le.expect1-T} require fewer conditions on $\theta$ and more on $a$ as compared to those leading to \eqref{eq:alpha.le.theta.2.gamma}. That is, though Stein coefficients may fail to exist for certain mean zero, variance one distributions, discrete ones in particular, the zero bias distribution here exists for all. We note that by Stein's characterization, when $a$ is standard normal we may take $T=1$ in \eqref{eq:alpha.le.expect1-T}, and $\gamma_a=0$ in \eqref{eq:alpha.le.theta.2.gamma}, and hence $\alpha=0$ in both the cases considered in the theorem that follows. The bound \eqref{eq:alpha.le.theta.2.gamma} also returns zero discrepancy in the special case where $\theta$ is linear, and thus recovers the results on linear compressed sensing \cite{RV08} when combined with Theorem \ref{main-theorem}.

For a real valued function $f$ with domain $D$ let
\beas
\|f\|=\sup_{x \in D}|f(x)|.
\enas 
 

\begin{theorem} \label{thm:alpha.via.sc}
Let $a$ be a mean zero, variance one random variable 
and set ${\bf a}=(a_1,\ldots,a_d)$ with $a_1,\ldots,a_d$ independent random variables distributed as $a$, and let $\alpha$ be as in \eqref{def:alpha}.

(a) If $\theta \in {\rm Lip}_1$ and $a$ has Stein coefficient $T$, then 
\bea \label{eq:alpha.le.expect1-T}
\alpha \le E|1-T|.
\ena

(b) If $\theta$ possesses a bounded second derivative, then 
\bea \label{eq:alpha.le.theta.2.gamma}
\alpha \le \|\theta''\| \gamma_a.
\ena 
\end{theorem}

\begin{remark} \label{rem:APPV14result}
In \cite{ALPV14} the quantity $\alpha$ is bounded in terms of the total variation distance $d_{\rm TV}(a,g)$ between $a$ and the standard Gaussian distribution $g$. In particular, for $\theta \in C^2$, Proposition 5.5 of \cite{ALPV14} yields 
\bea \label{eq:alpv14.alpha.bound}
\alpha \le 8(Ea^6+Eg^6)^{1/2}(\|\theta'\|+\|\theta''\|)\sqrt{d_{\rm TV}(a,g)}.
\ena

In contrast, the upper bound \eqref{eq:alpha.le.expect1-T} does not depend on any moments of $a$, requires $\theta$ to be only once differentiable, and in typical cases where $d_{\rm TV}(a,g)$ and $E|1-T|$ are of the same order, that is, when the upper bound in Lemma \ref{lem:dtv.bd.h} is of the correct order, $\alpha$ in \eqref{eq:alpha.le.expect1-T} is bounded by a first power rather than the larger square root in \eqref{eq:alpv14.alpha.bound}. 

When $\theta$ possesses a bounded second derivative, the upper bound \eqref{eq:alpha.le.theta.2.gamma} improves on \eqref{eq:alpv14.alpha.bound} in terms of constant factors, requirements on the existence of moments, and dependence on a first power rather than a square root. In this case Lemma \ref{lem:same.order.tv} shows $d_{\rm TV}(a,g)$ and $\gamma_a$ are of the same order when $a$ has bounded support 
\end{remark}

Measuring discrepancy from normality in terms of $E|1-T|$ and $\gamma_a$ also has the advantage of being tractable when each component of the Gaussian sensing vector ${\bf g}$ has been independently corrupted at the level of some $\epsilon \in [0,1]$ by a non Gaussian, mean zero, variance one distribution $a$. In the two models we consider we let the sensing vector have i.i.d.\ entries, and hence only specify the distribution of its components.
The first model is the case of additive error, where each component of the sensing vector is of the form
\bea \label{eq:geps.additive}
g_\epsilon = \sqrt{1-\epsilon}g+\sqrt{\epsilon} a
\ena
with $a$ independent of $g$, with the second one being the mixture model where each component has been corrupted due to some `bad event' $A$ that substitutes $g$ with $a$ so that
\bea \label{eq:geps.mixture}
g_\epsilon=g{\bf 1}_{A^c} + a{\bf 1}_A,
\ena  
where $A$ occurs with probability $\epsilon$, independently of $g,a$ and a given Stein coefficient $T$ for $a$. Since
\bea \label{eq:E[T|a]}
E[Tf'(a)]=E [E[T|a] f'(a)]
\ena
we see that $E[T|a]$ is a Stein coefficient for $a$. Hence, upon replacing $T$ by $E[T|a]$ only the independence of $A$ from $\{g,a\}$ is required.

Theorem \ref{thm:Lip.Error.Models} shows that under both scenarios (a) and (b) considered in Theorem \ref{thm:alpha.via.sc}, and further, under both the additive and mixture models, the value $\alpha$ can be bounded explicitly in terms of a quantity that vanishes in $\epsilon$. Further, we note that both error models agree with each other, and with the model of Theorem \ref{thm:alpha.via.sc}, when $\epsilon=1$, so that Theorem \ref{thm:Lip.Error.Models} recovers Theorem \ref{thm:alpha.via.sc} when so specializing. We now present Theorem \ref{thm:Lip.Error.Models} followed by its proof, then the proof of Theorem \ref{thm:alpha.via.sc}.

\begin{theorem}\label{thm:Lip.Error.Models}
Under condition (a) of Theorem \ref{thm:alpha.via.sc}, under both the additive \eqref{eq:geps.additive} and mixture \eqref{eq:geps.mixture} error models, we have
\beas
\alpha \le \epsilon  E|1-T|.
\enas
As regards the measure $\gamma_a$ in (b) of Theorem \ref{thm:alpha.via.sc}, under the additive error model \eqref{eq:geps.additive},
\bea \label{thm:zb.geps.additive}
\gamma_{g_\epsilon} \le \epsilon^{3/2} \gamma_a, \qmq{and when $\theta$ has a bounded second derivative,}
\alpha \le \epsilon^{3/2}\|\theta''\| \gamma_a,
\ena
and under the mixture error model \eqref{eq:geps.mixture},
\bea \label{thm:zb.geps.mixture}
\gamma_{g_\epsilon} \le \epsilon \gamma_a, \qmq{and when $\theta$ has a bounded second derivative,}
\alpha \le \epsilon \|\theta''\| \gamma_a.
\ena 
\end{theorem}

\noindent {\em Proof:}
By the assumptions of independence and on the mean and variance of $a$ and $g$, in both error models $g_\epsilon$ has mean zero and variance 1. As the components of the sensing vector are i.i.d.\ by construction, the hypotheses on ${\bf a}$ in Theorem \ref{thm:alpha.via.sc} holds.

First consider scenario (a) under the additive error model. If a random variable $W$ is the sum of two independent mean zero variables $X$ and $Y$ with finite variances and Stein coefficients $T_X$ and $T_Y$ respectively, then for any Lipshitz function $f$ one has
\begin{multline*}
E[Wf(W)]=E[(X+Y)f(X+Y)]= E[Xf(X+Y)]+E[Yf(X+Y)] \\
= E[T_Xf'(X+Y)]+E[T_Yf'(X+Y)]
=E[(T_X+T_Y)f'(X+Y)]\\ = E[T_Wf'(W)] \qmq{where $T_W=T_X+T_Y$,}
\end{multline*}
showing that Stein coefficients are additive for independents summands. In particular, now also using \eqref{eq:sc.linear}, we see that the Stein coefficient $T_\epsilon$ for $g_\epsilon$ in \eqref{eq:geps.additive} is given by 
$T_\epsilon = 1-\epsilon + \epsilon T$,
where $T$ is the given Stein coefficient for $a$. As $1-T_\epsilon=\epsilon(1-T)$, the first claim of the lemma follows by applying Theorem \ref{thm:alpha.via.sc}.

For the mixture model, by the independence between $A$ and $\{a,g,T\}$,
\begin{multline*} 
E[g_\epsilon f(g_\epsilon)]  = (1-\epsilon)E[gf(g)]+\epsilon E[af(a)] = (1-\epsilon)E[f'(g)]+\epsilon E[Tf'(a)] \\
=E[{\bf 1}_{A^c}f'(g)+T{\bf 1}_Af'(a)]= E[{\bf 1}_{A^c}f'(g_\epsilon)+T{\bf 1}_Af'(g_\epsilon)] = E[T_\epsilon f'(g_\epsilon)] \qmq{where} T_\epsilon = {\bf 1}_{A^c} + T{\bf 1}_A.
\end{multline*}
Hence the bound just shown for the additive model is seen to hold also for the mixture model by applying Theorem \ref{thm:alpha.via.sc} and observing that $1-T_\epsilon={\bf 1}_A(1-T)$ and recalling the independence between $T$ and $A$. 

Now consider scenario (b) under the additive error model.this paragraph rewritten for clarity Identity \eqref{eq:Y*.replace.one} says one may construct the zero bias distribution of a sum of independent terms by choosing a summand proportional to its variance and replacing it by a variable independent of the remaining summands and having the chosen summands' zero bias distribution, where the replacement is done independent of all else. As the two summands in \eqref{eq:geps.additive} have variance $1-\epsilon$ and $\epsilon$, we choose them for replacement with these probabilities, respectively. Hence, letting $B$ be the event that $a$ is chosen, we see \beas
g_\epsilon^*  = 
(\sqrt{1-\epsilon}g^*+\sqrt{\epsilon} a) {\bf 1}_{B^c} + (\sqrt{1-\epsilon}g+\sqrt{\epsilon} a^*){\bf 1}_B = \sqrt{1-\epsilon} g + \sqrt{\epsilon} \left( a {\bf 1}_{B^c} + a^* {\bf 1}_B\right),
\enas
has the $g_\epsilon$-zero bias distribution, where for the first equality we have applied \eqref{eq:zero.bias.homogeneous}, yielding $(\sqrt{1-\epsilon} g)^*=_d\sqrt{1-\epsilon} g^*$  and likewise $(\epsilon a)^*=_d\epsilon a^*$, and used that the standard normal is a fixed point of the zero bias transformation for the second. 
In addition, we construct $a^*$ to have the $a$-zero bias distribution, be independent of $g$ and $B$, and achieve the infimum $d_1({\cal L}(a),{\cal L}(a^*))$ in \eqref{eq:d1.inf}, that is, giving the coupling that minimizes $E|a^*-a|$.

We now obtain
\beas
g_\epsilon^*-g_\epsilon = \sqrt{1-\epsilon} g + \sqrt{\epsilon} \left( a {\bf 1}_{B^c} + a^* {\bf 1}_B\right)
-(\sqrt{1-\epsilon} g + \sqrt{\epsilon} a)=\sqrt{\epsilon}(a^*-a){\bf 1}_B.
\enas
As the Wasserstein distance is the infimum \eqref{eq:d1.inf} over all couplings between $g$ and $g_\epsilon$, using that $B$ is independent of $a$ and $a^*$, we have
\beas  
\gamma_{g_\epsilon} = d_1(g_\epsilon,g_\epsilon^*) \le E|g_\epsilon^*-g_\epsilon| 
= \sqrt\epsilon E|a^*-a|P(B) = \epsilon^{3/2}\gamma_a.
\enas
The proof of \eqref{thm:zb.geps.additive}, the first claim under (b), can now be completed by applying \eqref{eq:alpha.le.theta.2.gamma}.

Continuing under scenario (b), again consider the mixture model \eqref{eq:geps.mixture}. By Theorem 2.1 of \cite{G10}, as ${\rm Var}(a)={\rm Var}(g)$, the variable
\beas
g_\epsilon^*= g^*{\bf 1}_{A^c} + a^*{\bf 1}_A= g{\bf 1}_{A^c} + a^*{\bf 1}_A
\enas
has the $g_\epsilon$ zero bias distribution, where we again take $g^*$ and $a^*$ as in the previous construction. Hence, arguing as for the additive error model, we obtain the bound
\beas 
\gamma_{g_\epsilon} \le E|g_\epsilon^*-g_\epsilon| =\expect{|a^*-a|{\bf 1}_A} = \epsilon \gamma_a.
\enas
The second claim under (b) now follows as the first. 
\bbox

\vspace{1ex}

\noindent {\em Proof of Theorem \ref{thm:alpha.via.sc}.} Recalling that ${\bf x}$ is a unit vector, for any ${\bf t} \in B_2^d$ the vectors ${\bf x}$ and ${\bf v}={\bf t}-\langle {\bf x},{\bf t} \rangle {\bf x}$ are perpendicular. If ${\bf v}\not =0$ set ${\bf x}^\perp$ to be the unit vector in direction ${\bf v}$, and let ${\bf x}^\perp$ be zero otherwise.
These vectors produce an orthogonal decomposition of any ${\bf t} \in B_2^d$ as 
\bea \label{eq:decomposition}
{\bf t}= \dotp{{\bf x}}{{\bf t}} {\bf x} + \dotp{{\bf x}^\perp}{{\bf t}} {\bf x}^\perp.
\ena
Defining
\beas
Y=\langle {\bf a},{\bf x} \rangle \qmq{and} Y^\perp=\dotp{ {\bf a}}{{\bf x}^\perp},
\enas
using the decomposition \eqref{eq:decomposition} in \eqref{eq:Efxx.to.expand}, and the expression for $\lambda$ in \eqref{def:altlambda}
yields
\begin{multline*} 
E[f_{{\bf x}}({\bf t})] = E[\dotp{{\bf a}}{ {\bf t}}\theta(\dotp{{\bf a}}{ {\bf x}})]  = \dotp{{\bf x}}{{\bf t}} E\left[\dotp{{\bf a}}{{\bf x}} \theta(\dotp{{\bf a}}{ {\bf x}})\right]+\dotp{ {\bf x}^\perp}{{\bf t} } E\left[\dotp{ {\bf a}}{{\bf x}^\perp } \theta(\dotp{ {\bf a}}{ {\bf x}})\right] \\
=  \lambda \dotp{{\bf x}}{{\bf t}} +\dotp{ {\bf x}^\perp}{{\bf t} } E\left[Y^\perp \theta(Y)\right].
\end{multline*}
As $\left\|\mathbf{x}^\perp\right\|_2$ and $\|{\bf t}\|_2$ are at most one, applying the Cauchy-Schwarz inequality we obtain moved from below to be used in both cases
\bea \label{eq:expectfxx'}
\left|E[f_{{\bf x}}({\bf t})]-\lambda \dotp{ {\bf x}}{{\bf t} }\right| \le \Bvert E\left[Y^\perp \theta(Y)\right] \Bvert.
\ena

We determine a Stein coefficient for $Y^\perp$
as follows. For $T_i$ Stein coefficients for $a_i$, independent and identically distributed as the given $T$ for all $i=1,\ldots,d$, by conditioning on $Y-x_ia_i$, a function of $\{a_j, j \not = i\}$ and therefore independent of $a_i$, using the scaling property \eqref{eq:sc.linear} we have
\begin{multline} \label{eq:branch.zb.from.here}
E[x_i^\perp a_i\theta(Y)]=E\left[ x_i^\perp a_i \theta(x_i a_i + (Y-x_ia_i))\right]=E\left[ x_i^\perp x_i T_i \theta'(x_i a_i + (Y-x_ia_i))\right]\\
=E\left[ x_i^\perp x_i T_i \theta'(Y)\right].
\end{multline}
Hence
\begin{multline} 
E[Y^\perp \theta(Y)]=\sum_{i=1}^d E[x_i^\perp a_i \theta(Y)]=E[T_{Y^\perp} \theta'(Y)] \\
\mbox{where} \quad T_{Y^\perp}= \sum_{i=1}^d x_i^\perp x_iT_i= \sum_{i=1}^d x_i^\perp x_i(T_i-1),\label{eq:hperp} 
\end{multline}
where the last equality follows from $\dotp{\mathbf{x}}{\mathbf{x}^\perp}=0$.


Now from  \eqref{eq:expectfxx'} and \eqref{eq:hperp} we have
\begin{multline*} 
\left|E[f_{{\bf x}}({\bf t})]-\lambda \dotp{ {\bf x}}{{\bf t} }\right|
\le   |E[T_{Y^\perp} \theta'(Y)]|\\
\le E|T_{Y^\perp}| \le \sum_{i=1}^d |x_i^\perp x_i | E|T-1|
 \le\|\mathbf{x}^\perp\|_2\|\mathbf{x}\|_2 E|T-1|\le E|T-1|,
\end{multline*}
using $\theta \in {\rm Lip}_1$ in the second inequality, followed by \eqref{eq:hperp} again
and the Cauchy-Schwarz inequality, noting that $\|\mathbf{x}^\perp\|_2 \le 1$ and $\|\mathbf{x}\|_2=1$. Hence we obtain
\beas 
\left|E[f_{{\bf x}}({\bf t})]-\lambda \dotp{ {\bf x}}{{\bf t} }\right|
\le  E|T-1| \qmq{for all ${\bf t} \in B_2^d$,}
\enas
which completes the proof of \eqref{eq:alpha.le.expect1-T} in light of the definition \eqref{def:alpha} of $\alpha$.

In a similar fashion, if $\theta$ is twice differentiable with bounded second derivative, then in place of \eqref{eq:branch.zb.from.here}, for every $i=1,\ldots,d$ we may write
\beas
E[x_i^\perp a_i\theta(Y)]=E\left[ x_i^\perp a_i \theta(x_i a_i + (Y-x_ia_i))\right]=E\left[ x_i^\perp x_i  \theta'(x_i a_i^* + (Y-x_ia_i))\right],
\enas
where $a_i,a_i^*$ are constructed on the same space to be an optimal coupling, in the sense of achieving the infimum of $E|a^*-a|$. Hence,
\begin{multline}\label{eq:Yperp.theta.zb.smthlip}
E[Y^\perp \theta(Y)]=\sum_{i=1}^d E[x_i^\perp a_i \theta(Y)]=\sum_{i=1}^d E\left[ x_i^\perp x_i  \theta'(x_i a_i^* + (Y-x_ia_i))\right] \\
= \sum_{i=1}^d E\left[ x_i^\perp x_i  (\theta'(x_i a_i^* + (Y-x_ia_i)) - \theta'(Y)) \right] \\= 
\sum_{i=1}^d E\left[ x_i^\perp x_i  (\theta'(x_i a_i^* + (Y-x_ia_i)) - \theta'(x_i a_i + (Y-x_ia_i))) \right],
\end{multline}
where in the third inequality we have used $\langle {\bf x}^\perp,{\bf x} \rangle =0$, as in 
\eqref{eq:hperp}.

The proof of \eqref{eq:alpha.le.theta.2.gamma} is completed by applying \eqref{eq:expectfxx'} and \eqref{eq:Yperp.theta.zb.smthlip} to obtain
\begin{multline*}
\left|E[f_{{\bf x}}({\bf t})]-\lambda \dotp{ {\bf x}}{{\bf t} }\right| \le \left| E[Y^\perp \theta(Y)] \right| \le \|\theta''\| \sum_{i=1}^d E \left| x_i^\perp x_i^2  (a_i^*-a_i) \right|  \le \|\theta''\| \gamma_a \sum_{i=1}^d  \left| x_i^\perp x_i^2 \right| \\
\le 
\|\theta''\| \gamma_a \sum_{i=1}^d \left| x_i^\perp x_i \right| \le \|\theta''\| \gamma_a,
\end{multline*}
where we have applied the mean value theorem for the second inequality,  the fact that the infimum in \eqref{eq:d1.inf} is achieved for the third, 
that $\|{\bf x}\|_2=1$ for the fourth, and the Cauchy-Schwarz inequality for the last.
\bbox

\subsection{Sign function}\label{sec:sign}
In this section we consider the case where $\theta$ is the sign function given by
\beas
\theta(x)=\left\{
\begin{array}{rc}
-1 & x <0\\
1 & x \ge 0.
\end{array}
\right.
\enas
The motivation comes from the one bit compressed sensing model, see \cite{ALPV14} for a more detailed discussion.
The following result shows
how $\alpha$ of \eqref{def:alpha} can be bounded in terms of the discrepancy measure $\gamma_a$ introduced in Section \ref{sec:lip}. 
Throughout this section set 
\beas 
c_1=\sqrt{2/\pi}-1/2.
\enas
 We continue to assume that the unknown vector ${\bf x}$ has unit Euclidean length. 
 
In the following, we say a random variable $a$ is symmetric if the distributions of $a$ and $-a$ are equal.
\begin{theorem}\label{thm:sign.zero.bias} Let $\theta$ be the sign function, $a$ have a symmetric distribution, and $\gamma_a$  as defined in \eqref{def:gammaa}. If $\|{\bf x}\|_3^3 \le c_1/\gamma_a$ and $\|{\bf x}\|_\infty \le 1/2$, then $\alpha$ defined in \eqref{def:alpha} satisfies
	\bea \label{eq:Lemma4.3.improved}
	\alpha \le \left(10 \gamma_a E|a|^3 \|{\bf x}\|_\infty \right)^{1/2}.
	\ena
\end{theorem}

Under the condition that $\|{\bf x}\|_\infty \le c/E|a|^3$ for some $c>0$, Proposition 4.1  of \cite{ALPV14} yields the existence of a constant $C$ such that
\bea \label{eq:Lemma4.3}
\alpha \le CE|a|^3 \|{\bf x}\|_\infty^{1/2}.
\ena
Theorem \ref{thm:sign.zero.bias} improves \eqref{eq:Lemma4.3} by introducing the factor of $\gamma_a$ in the bound, thus providing a right hand side that takes the value 0 when $a$ is normal. 
Applying the inequality $\gamma_a \le E|a|^3/2$ in \eqref{eq:Edifasax3} to \eqref{eq:Lemma4.3.improved} in the case where $a$ has finite third moment recovers \eqref{eq:Lemma4.3} with $C$ assigned the specific value of $\sqrt{5}$. 

In terms of the total variation distance between $a$ and the Gaussian $g$, Proposition 5.2 in \cite{ALPV14} provides the bound
\beas
\alpha \le C (Ea^4)^{1/8} d_{\rm TV}(a,g)^{1/8}
\enas
depending on an unspecified constant and an eighth root. For distributions where $\gamma_a$ is comparable to the total variation distance, see Section \ref{sec:SC.ZB.relation}, the bound of Theorem \ref{thm:sign.zero.bias} would be preferred as far as its dependence on the distance between $a$ and $g$, and is also explicit.

Now we derive bounds on $\alpha$ defined in \eqref{def:alpha} for the two error models introduced in Section \ref{sec:lip}. As in Theorem \ref{thm:Lip.Error.Models}, the bounds vanish as $\epsilon$ tends to zero. We note that Theorem \ref{thm:sign.zero.bias} is recovered as the special case $\epsilon=1$ for both models considered. For comparison, in view of the relation between \eqref{eq:Lemma4.3.improved} of Theorem \ref{thm:sign.zero.bias} and \eqref{eq:Lemma4.3}, for these error models the bounds one obtains from the latter are the same as the ones below, but with the factor $\gamma_a$ replaced by $C=\sqrt{5}$ by virtue of \eqref{eq:Edifasax3}, and with the cubic term, which gives a bound on the third absolute moment of the $\epsilon$-contaminated distribution, appearing outside the square root. 

\begin{theorem}\label{thm:Sign.Error.Models}
In the additive and mixture error models \eqref{eq:geps.additive} and  \eqref{eq:geps.mixture}, the bound of Theorem \ref{thm:sign.zero.bias} becomes, respectively
\beas
\alpha \le \left(10\epsilon^{3/2}\gamma_a\left(\sqrt{1-\epsilon}\left(\sqrt{\frac8\pi}\right)^{1/3}+\sqrt{\epsilon}\expect{|a|^3}^{1/3}\right)^3\|\mathbf{x}\|_{\infty}\right)^{1/2} 
\enas
and
\beas
\alpha \le \left(10\epsilon\gamma_a\left(\left((1-\epsilon) \sqrt{\frac8\pi}\right)^{1/3}+\expect{\epsilon|a|^3}^{1/3}\right)^3\|\mathbf{x}\|_{\infty}\right)^{1/2}.
\enas
\end{theorem}

We first demonstrate the proof of Theorem \ref{thm:sign.zero.bias}, starting with a series of lemmas. 

\begin{lemma} \label{lem:signvxx.lambda}
For any mean zero, variance 1 random variable $a$, and any $\mathbf{x} \in B_2^d$ ,
\bea \label{ineq:signvxx.lambda}
\left| \dotp{ {\bf v}_{\bf x}}{ {\bf x}} - \sqrt{\frac2\pi} \right|\le \gamma_a \|{\bf x}\|_3^3,
\ena
where ${\bf v}_{\bf x}=E[{\bf a}\theta(\langle {\bf a}, {\bf x} \rangle )]$ as in \eqref{eq:Efxx.to.expand}.
\end{lemma}

The inequality in Lemma \ref{lem:signvxx.lambda} should be compared to Lemma 5.3 of \cite{ALPV14}, where the bound on the quantity in \eqref{ineq:signvxx.lambda} is in terms of the fourth root of the total variation distance between $a$ and $g$ and their fourth moments. \\

\noindent {\em Proof:} It is direct to verify that $E|g|=\sqrt{2/\pi}$ for $g \sim {\cal N}(0,1)$. In Lemma \ref{lem:stein.sol.abs} in Appendix \ref{app:additional}, we show that when taking $f$ to be the unique bounded solution to the Stein equation
\begin{equation}\label{stein's-equation}
f'(x)-xf(x)=|x|-\sqrt{\frac{2}{\pi}},
\end{equation}
we have $\|f''\|_\infty =1$, where $\|\cdot\|_\infty$ is the essential supremum. Hence for a mean zero, variance one random variable $Y$, using that sets of measure zero do not affect the integral below, we have
\begin{multline*}
|E|Y|-E|g||=|E[f'(Y)-Yf(Y)]|=|E[f'(Y)-f'(Y^*)]| =\left| \expect{\int_Y^{Y^*} f''(u)du}\right|
\\
\le \|f''\|_{\infty} E|Y^*-Y|=E|Y^*-Y|,
\end{multline*}
where $Y^*$ is any random variable on the same space as $Y$, having the $Y$-zero biased distribution.

As $\theta$ is the sign function
\beas
\dotp{ {\bf v}_{\bf x}}{{\bf x}} = E[\langle {\bf a}, {\bf x} \rangle \theta( \langle {\bf a}, {\bf x} \rangle )  ]=E|\langle {\bf a}, {\bf x} \rangle|
\qmq{and hence} \left|\dotp{ {\bf v}_{\bf x}}{ {\bf x}} - \sqrt{\frac2\pi}\right|=|E|\langle {\bf a}, {\bf x} \rangle|- E|g||.
\enas

For the case at hand, let $Y=\langle {\bf a}, {\bf x} \rangle = \sum_{i=1}^n x_i a_i$, where $a_1,\ldots,a_n$ are independent and identically distributed as $a$, having mean zero and variance 1 and recall $\|{\bf x}\|_2=1$. Then with $P[I=i]=x_i^2$, taking $(a_i,a_i^*)$ to achieve the infimun in  \eqref{eq:d1.inf}, that is, so 
that $E|a_i^*-a_i|=d_1(a_i,a_i^*)$, by \eqref{eq:Y*.replace.one} we obtain
\beas
E|Y^*-Y|=E|x_I(a_I^*-a_I)| = \sum_{i=1}^n 
|x_i|^3 \gamma_{a_i} = \gamma_a \|{\bf x}\|_3^3,
\enas 
as desired.
\bbox

We now provide a version of Lemma 4.4 of \cite{ALPV14} in terms of $\gamma_a$ and specific constants. 

\begin{lemma} \label{lem:Lemma4.4}
The vector ${\bf v}_{\bf x}$ in \eqref{eq:Efxx.to.expand} satisfies $\|{\bf v}_{\bf x}\|_2 \le 1$, and if $\|{\bf x}\|_3^3 \le c_1/\gamma_a$ where  $c_1=\sqrt{2/\pi}-1/2$, then 
$$
\frac{1}{2} \le \|{\bf v}_{\bf x}\|_2.
$$	
\end{lemma}

\noindent {\em Proof:} The upper bound follows as in the proof Lemma 4.4 in \cite{ALPV14}. Slightly modifying the lower bound argument there through the use of Lemma \ref{lem:signvxx.lambda} for the  second inequality below we obtain 
\beas
\|{\bf v}_{\bf x}\|_2 = \|{\bf v}_{\bf x}\|_2 \|{\bf x}\|_2 \ge |\langle {\bf v}_{\bf x},{\bf x} \rangle | \ge \sqrt{\frac2\pi} - \gamma_a \|{\bf x}\|_3^3 \ge \sqrt{\frac2\pi} -c_1= 1/2.
\enas
\bbox


Next we provide a version of Lemma 4.5 of \cite{ALPV14} with the explicit constant 2, following the proof there, and impose a symmetry assumption on $a$ that was used implicitly.
\begin{lemma} \label{lem:Lemma4.5}
	If $\|{\bf x}\|_\infty \le 1/2$ and $a$  has a symmetric distribution then the vector ${\bf v}_{\bf x}$ in \eqref{eq:Efxx.to.expand} satisfies
	\beas 
	\|{\bf v}_{\bf x}\|_\infty \le 2 E|a|^3 \|{\bf x}\|_\infty. 
	\enas
\end{lemma}

\noindent {\em Proof:}  By the symmetry of $a$ we assume without loss of generality that $x_j \ge 0$ for all $j=1,\ldots,d$ when considering the inner product $S=\langle {\bf a},{\bf x} \rangle$. For a given coordinate index $i$ let $S^{(i)}=\langle {\bf a},{\bf x} \rangle - a_i x_i$. Using symmetry again in the second equality below and setting $\tau_i^2 = \sum_{k \not = i}x_k^2$, for fixed $r \ge 0$ we obtain
\begin{multline*}
|E\theta(S^{(i)}+rx_i)| = |P[S^{(i)} \ge -rx_i]-P[S^{(i)} < -rx_i]| \\
= |P[S^{(i)} \ge -rx_i]-P[S^{(i)} > rx_i]| = P[|S^{(i)}| \le rx_i] = P[|S^{(i)}|/\tau_i \le rx_i/\tau_i] \\
\le P[|g| \le rx_i/\tau_i] + |P[|S^{(i)}|/\tau_i \le rx_i/\tau_i]-P[|g| \le rx_i/\tau_i]|.
\end{multline*}
The hypothesis $\|{\bf x}\|_\infty \le 1/2$ implies $\tau_i^2 \ge 3/4$. Hence,  
using the supremum bound on the standard normal density for the first term and that $\sqrt{8/3\pi} \le 1$, the Berry-Esseen bound of \cite{Sh10} with constant $0.56$ on the second term, noting $0.56 (4/3)^{3/2} \le 1$, and that $\|{\bf x}\|_3^3 \le \|{\bf x}\|_\infty$ since $\|{\bf x}\|_2=1$, we obtain
\beas
|E[r\theta(S^{(i)}+rx_i)]| \le r^2 x_i + |r| \|{\bf x}\|_\infty  E|a|^3 .
\enas

Considering now the $i^{th}$ coordinate of ${\bf v}_{\bf x}=E[ {\bf a} \theta(\langle {\bf a},{\bf x} \rangle)]$, using $E|a| \le (Ea^2)^{1/2}=1 \le (E|a|^3)^{1/3} \le E|a|^3$, we have
\beas
|E[ a_i \theta(\langle {\bf a},{\bf x} \rangle)]| = |E[a_i \theta(S^{(i)}+a_i x_i)]| \le x_i + \|{\bf x}\|_\infty E|a|^3 
\le 2E|a|^3 \|{\bf x}\|_\infty.
\enas
A similar computation yields this same result when $r<0$.
\bbox

\noindent {\em Proof of Theorem \ref{thm:sign.zero.bias}:} We follow the proof of Proposition 4.1 of \cite{ALPV14}. 
By Lemma \ref{lem:Lemma4.4} we see ${\bf v}_{\bf x} \not =0$, and defining ${\bf z}={\bf v}_{\bf x}/\|{\bf v}_{\bf x}\|_2$, from Lemmas \ref{lem:Lemma4.4} and \ref{lem:Lemma4.5}
\beas
\|{\bf z}\|_\infty=
\frac{\|{\bf v}_{\bf x}\|_\infty}{\|{\bf v}_{\bf x}\|_2} \le 2\|{\bf v}_{\bf x}\|_\infty \le 4 E|a|^3 \|{\bf x}\|_\infty.
\enas

Hence, first using the triangle inequality together with the fact that $|\theta(\cdot)| = 1$, with the equality following holding because $\theta$ is the sign function, and the second inequality following from Lemma \ref{lem:signvxx.lambda}, we obtain
\begin{multline}\label{inter-bound-on-norm}
\|{\bf v}_{\bf x}\|_2 = \langle {\bf v}_{\bf x}, {\bf z}\rangle = 
E[\theta(\langle {\bf a}, {\bf x}\rangle)\langle {\bf a},{\bf z} \rangle ]  \le E[|\langle {\bf a},{\bf z} \rangle |] = E[\theta(\langle {\bf a}, {\bf z}\rangle)\langle {\bf a},{\bf z} \rangle ]
\le  \sqrt{\frac2\pi} +\gamma_a \|{\bf z}\|_\infty \\
\le  \sqrt{\frac2\pi} + 4\gamma_a E|a|^3 \|{\bf x}\|_\infty.
\end{multline}
Next, using \eqref{eq:Efxx.to.expand}, we bound $|E[f_{\mathbf{x}}(\mathbf{t})]-\lambda\dotp{\mathbf{x}}{\mathbf{t}}|=|\dotp{ {\bf v}_{\bf x}}{{\bf t}} - \lambda \dotp{ {\bf x}}{{\bf t} }|$. By the Cauchy-Schwartz inequality, now taking ${\bf t} \in B_2^d$,
\begin{align*}
\left|\dotp{ {\bf v}_{\bf x}}{ {\bf t}} - \lambda \dotp{ {\bf x}}{{\bf t} } \right|^2 = \left|\dotp{ {\bf v}_{\bf x}-\lambda{\bf x}}{ {\bf t} } \right|^2\le \|{\bf v}_{\bf x}-\lambda{\bf x}\|^2. 
\end{align*}
Furthermore, by  \eqref{def:altlambda}, we have $\dotp{{\bf v}_{\mathbf{x}}}{\mathbf{x}}=\lambda$, thus
\begin{multline*}
\|{\bf v}_{\bf x}-\lambda{\bf x}\|^2= \|{\bf v}_{\bf x}\|_2^2 -\lambda^2 + 2\lambda (\lambda- \langle {\bf v}_{\bf x},{\bf x}\rangle )
= (\|{\bf v}_{\bf x}\|_2 -\lambda)(\|{\bf v}_{\bf x}\|_2 +\lambda)
\leq2(\|{\bf v}_{\bf x}\|_2 -\lambda)\\
=2\left(\|{\bf v}_{\bf x}\|_2 -\sqrt{\frac2\pi}+\sqrt{\frac2\pi}-\lambda\right)
\le 10 \gamma_a E|a|^3 \|{\bf x}\|_\infty,
\end{multline*}
where we have applied Lemma \ref{lem:Lemma4.4} in the first inequality and the last inequality follows from  \eqref{inter-bound-on-norm}, Lemma \ref{lem:signvxx.lambda} and that $E|a|^3 \ge 1$. Now taking a square root finishes the proof. \bbox
\vspace{0.5cm}

\noindent {\em Proof of Theorem \ref{thm:Sign.Error.Models}:} 
Under the additive error model \eqref{eq:geps.additive}, by Minkowski's inequality 
\begin{multline*}
\expect{|g_\epsilon|^3}^{1/3}=\expect{|\sqrt{1-\epsilon}g+\sqrt{\epsilon} a|^3}^{1/3}
\leq\sqrt{1-\epsilon}\expect{|g|^3}^{1/3}+\sqrt{\epsilon}\expect{|a|^3}^{1/3}\\
=\sqrt{1-\epsilon}\left(\sqrt{\frac8\pi}\right)^{1/3}+\sqrt{\epsilon}\expect{|a|^3}^{1/3}.
\end{multline*}
Using this inequality and \eqref{thm:zb.geps.additive} in Theorem \ref{thm:sign.zero.bias} gives the discrepancy bound in the additive error case.

For the mixture model \eqref{eq:geps.mixture}, again by Minkowski's inequality, 
\begin{multline*}
\expect{|g_\epsilon|^3}^{1/3}=\expect{|g{\bf 1}_{A^c} + a{\bf 1}_A|^3}^{1/3}
\leq\expect{(1-\epsilon)|g|^3}^{1/3}+\expect{|\epsilon a|^3}^{1/3}\\
=\left((1-\epsilon)\sqrt{\frac8\pi}\right)^{1/3}+\expect{\epsilon |a|^3}^{1/3}.
\end{multline*}
Using this inequality and \eqref{thm:zb.geps.mixture} in Theorem \ref{thm:sign.zero.bias} gives the discrepancy bound in the mixed error case.
\bbox

\subsection{Relations between Measures of Discrepancy} \label{sec:SC.ZB.relation}
We have considered two methods for handling non-Gaussian sensing, the first using Stein coefficients and the second by the zero bias distribution. In this section we discuss some relations between these two, and also their connections to the total variation distance $d_{{\rm TV}}(\cdot,\cdot)$ appearing in the bound of \cite{ALPV14} and discussed in Remark \ref{rem:APPV14result}. 

The following result appears in Section 7 of \cite{Go07}.
\begin{lemma}\label{lem:dtv.zb}
If $a$ is a mean zero, variance 1 random variable, and $a^*$ has the $a$-zero biased distribution, then
\bea \label{eq:tvyb.le.yys}
d_{\rm TV}(a,g) \le 2 d_{\rm TV}(a,a^*).
\ena
\end{lemma}

The following related result is from \cite{Ch09}.
\begin{lemma} \label{lem:dtv.bd.h}
	If the mean zero, variance 1 random variable $a$ has Stein coefficient $T$, then 
	\beas
	d_{\rm TV}(a,g) \le 2 E|1-T|,
	\enas
	where $g \sim {\cal N}(0,1)$.
\end{lemma}

Since $E[Tf'(a)]=E[E[T|a]f'(a)]$, if $T$ is a Stein coefficient for $a$ then so is $h(a)=E[T|a]$. 
Introducing this Stein coefficient in the identity that characterizes the zero bias distribution $a^*$, we obtain
\beas
E[f'(a^*)]=E[af(a)]=E[h(a)f'(a)].
\enas
Hence, when such a $T$ exists $h(a)$ is the Radon Nikodym derivative of 
the distribution of $a^*$ with respect to that of $a$, and in particular ${\cal L}(a^*)$ is absolutely continuous with respect to ${\cal L}(a)$. When $a$ is a mean zero, variance one random variable with density $p(a)$ whose support is a possibly infinite interval, then using the form of the density $p^*(a)$ of $a^*$ as given in \cite{GR97}, we have
\bea \label{eq:h.in.terms.p}
p^*(y)=E[a{\bf 1}(a>y)] \qmq{and}
h(y)=\frac{p^*(y)}{p(y)}{\bf 1}(p(y)>0)=\frac{E[a{\bf 1}(a>y)]}{p(y)}{\bf 1}(p(y)>0),
\ena
and hence, 
\beas
E|1-h(a)| = \int_{y:p(y)>0} \Bvert 1-\frac{p^*(y)}{p(y)}\Bvert p(y)dy = \int_{\mathbb{R}}|p(y)-p^*(y)|dy=d_{\rm TV}(a,a^*),
\enas
and the upper bounds in Lemmas \ref{lem:dtv.bd.h} and  \ref{lem:dtv.zb} are equal. Overall then, in the case where the Stein cofficient of a random variable is given as a function of the random variable itself, the discrepancy measure considered in Theorem \ref{thm:Lip.Error.Models} under part (a) of Theorem \ref{thm:alpha.via.sc} is simply the total variation distance between $a$ and $a^*$, while that under part (b), and in Section \ref{sec:sign} when $\theta(\cdot)$ is specialized to be the sign function, is the Wasserstein distance.

Due to a result of \cite{Ch09}, 
Stein coefficients can be constructed in some generality when $a=F({\bf g})$ for some differentiable function $F:\mathbb{R}^n \rightarrow \mathbb{R}$ of a standard normal vector ${\bf g}$ in $\mathbb{R}^n$. In this case 
\beas
T=\int_0^\infty e^{-t} \langle \nabla F({\bf g}),\widehat{E}(\nabla F({\bf g}_t))\rangle dt
\enas
is a Stein coefficient for $a$ where ${\bf g}_t=e^{-t}{\bf g}+\sqrt{1-e^{-2t}}\widehat{\bf g}$,
with $\widehat{\bf g}$ an independent copy of ${\bf g}$, and $\widehat{E}$ integrating over $\widehat{\bf g}$, that is, taking conditional expectation with respect to ${\bf g}$.

To provide a concrete example of a Stein coefficient, a simple computation using the final equality of \eqref{eq:h.in.terms.p} shows that
if $a$ has the double exponential distribution with variance 1, that is, with density 
\beas
p(y)=\frac{1}{\sqrt{2}}e^{-\sqrt{2}|y|} \qmq{then} h(y)=\frac{1}{2}(1+\sqrt{2}|y|).
\enas 
In this case 
\beas
E|1-h(a)|=E|1-\sqrt{2} a|{\bf 1}(a>0)= \frac{1}{e}.
\enas

The following result provides a bound complementary to \eqref{eq:tvyb.le.yys} of Lemma \ref{lem:dtv.zb}, which when taken together show that $d_{\rm TV}(a,a^*)$ and $d_{\rm TV}(a,g)$ are of the same order in general for distributions of bounded support.

\begin{lemma} \label{lem:same.order.tv}
If $a$ is a mean zero, variance one random variable with density $p(y)$ supported in $[-b,b]$ then
\beas
d_{\rm TV}(a,a^*) \le (1+b^2)d_{\rm TV}(a,g).
\enas
\end{lemma}

\noindent {\bf Proof:} With $p^*(y)$ the density of $a^*$ given by \eqref{eq:h.in.terms.p}, we have
\beas
d_{\rm TV}(a,a^*)=\int_{[-b,b]}|p(y)-p^*(y)|dy = \int_{[-b,b]}(p(y)-p^*(y))\phi(y)dy = E\phi(a)-E\phi(a^*),
\enas
where
\beas
\phi(y)= \left\{\
\begin{array}{rc}
	1 &  p(y) \ge p^*(y)\\
	-1 & p(y) < p^*(y).
\end{array}
\right.
\enas
Setting
\beas
f(y) = \int_0^y \phi(u) du \qmq{and} q(y)=\phi(y)-y \int_0^y \phi(u) du,
\enas
we have $f'(y)=\phi(y)$, and using \eqref{eq:zb.char} to yield $E[q(g)]=0$ we obtain
\beas
d_{\rm TV}(a,a^*)=E[f'(a)-f'(a^*)] =E[f'(a)-af(a)]  = Eq(a)-Eq(g).
\enas
For $y \in [-b,b]$ we have $|q(y)| \le |\phi(y)| + |y| \int_0^y |\phi(u)|du \le 1+b^2$, hence 
\beas 
d_{\rm TV}(a,a^*) \le (1+b^2) d_{\rm TV}(a,g),
\enas
as claimed. 
\bbox

\section{Proof of Theorem \ref{main-theorem}}\label{sec:proof.of.main.thm}

So far, we have shown that the penalty $\alpha$ for non-normality in \eqref{def:alpha} of Theorem \ref{main-theorem} can be bounded explicitly using discrepancy measures that arise in Stein's method. 
 In this section, we focus on proving Theorem \ref{main-theorem} via a generic chaining argument that is the crux to the concentration inequality applied.
 
Recall that 
 by \eqref{empirical-loss}, \eqref{estimator} and \eqref{def-empirical-function}, 
\[\widehat{\mathbf{x}}_m=\argmin_{\mathbf{t}\in K}~\left(\|\mathbf{t}\|_2^2-2f_{\mathbf{x}}(\mathbf{t})\right).\]

In order to demonstrate that $\widehat{\mathbf{x}}_m$ is a good estimate of $\lambda\mathbf{x}$,
we need to control the mean of $f_{\mathbf{x}}(\cdot)$ in \eqref{def-empirical-function}, and the deviation of $f_{\mathbf{x}}(\cdot)$ from its mean. As shown is the previous section, the mean of $f_{\mathbf{x}}(\cdot)$ can be effectively characterized through the introduced discrepancy measures. The deviation is controlled by the following lemma.
 
\begin{lemma}[Concentration]\label{concentration-of-functions}
 	Let $\mathcal T := D(K,\lambda\mathbf{x})\cap\mathbb{S}^{d-1}$. Under the assumptions of Theorem \ref{main-theorem},
 	for all $u\geq2$ and $m\geq \omega(\mathcal T)^2$,
 	\begin{align*}
 	P\left[\sup_{\mathbf{t}\in \mathcal T}\left|f_{\mathbf{x}}(\mathbf{t})-\expect{f_{\mathbf{x}}(\mathbf{t})}\right|
 	\geq C_0(\|a\|_{\psi_2}^2+\|y\|_{\psi_2}^2)\frac{\omega(\mathcal T)+u}{\sqrt{m}}\right]
 	\leq 4e^{-u},
 	\end{align*}
 	where $C_0>0$ is a fixed constant\footnote{Since the set $K$ is closed, the set $D(K,\lambda\mathbf{x})\cap\mathbb{S}^{d-1}\subseteq\mathbb{R}^d$ is also closed and thus Borel measurable. By taking $\mathcal T=D(K,\lambda\mathbf{x})\cap\mathbb{S}^{d-1}\subseteq\mathbb{R}^d$ in Remark \ref{remark-measurability},
 		we have that the supremum is indeed measurable in the probability space $(\Omega,~\mathcal{E},~P)$.}.
 \end{lemma}

 The proof of this lemma, provided in the next subsection, is based on the improved chaining technique introduced in \cite{tail-bound-chaining}. 
 We now show that once Lemma \ref{concentration-of-functions} is proved how Theorem \ref{main-theorem} follows without much 
 overhead.

 Using Lemma \ref{lem:Lt.minus.Llx} for the first inequality, we have
 \begin{align*}
 \|\widehat{\mathbf{x}}_m-\lambda\mathbf{x}\|_2^2\leq& L(\widehat{\mathbf{x}}_m)-L(\lambda\mathbf{x})
 +2\alpha\|\widehat{\mathbf{x}}_m-\lambda\mathbf{x}\|_2\\
 = &L(\widehat{\mathbf{x}}_m)-L_m(\widehat{\mathbf{x}}_m)+L_m(\widehat{\mathbf{x}}_m)
 -L_m(\lambda\mathbf{x})+L_m(\lambda\mathbf{x})-L(\lambda\mathbf{x})
 +  2\alpha\|\widehat{\mathbf{x}}_m-\lambda\mathbf{x}\|_2\\
 =& -2\left( E_m[y\dotp{\mathbf{a}}{\widehat{\mathbf{x}}_m}]-f_{\mathbf{x}}(\widehat{\mathbf{x}}_m)\right) +L_m(\widehat{\mathbf{x}}_m)-L_m(\lambda\mathbf{x})
 +2\left(
 E_m[y\dotp{\mathbf{a}}{\lambda\mathbf{x}}]
 -f_{\mathbf{x}}(\lambda\mathbf{x})\right)\\
 &+2\alpha\|\widehat{\mathbf{x}}_m-\lambda\mathbf{x}\|_2\\
 \leq&2\left|f_{\mathbf{x}}(\widehat{\mathbf{x}}_m-\lambda\mathbf{x})-E_m[y\dotp{\mathbf{a}}{\widehat{\mathbf{x}}_m-\lambda\mathbf{x}}]
 \right|
 +L_m(\widehat{\mathbf{x}}_m)-L_m(\lambda\mathbf{x}) +2\alpha\|\widehat{\mathbf{x}}_m-\lambda\mathbf{x}\|_2,
 \end{align*}
 where $E_m[\cdot]$ is the conditional expectation given $\{(\mathbf{a}_i,y_i)\}_{i=1}^m$. Since 
 $\widehat{\mathbf{x}}_m$ solves \eqref{estimator} and $\lambda\mathbf{x}\in K$, it follows that $L_m(\widehat{\mathbf{x}}_m)-L_m(\lambda\mathbf{x})\leq0$. Thus, 
 \[\|\widehat{\mathbf{x}}_m-\lambda\mathbf{x}\|_2^2
 \leq
 2\left|f_{\mathbf{x}}(\widehat{\mathbf{x}}_m-\lambda\mathbf{x})-E_m[y\dotp{\mathbf{a}}{\widehat{\mathbf{x}}_m-\lambda\mathbf{x}}]\right|
 +2\alpha\|\widehat{\mathbf{x}}_m-\lambda\mathbf{x}\|_2.\]
 Since $\widehat{\mathbf{x}}_m-\lambda\mathbf{x}\in D(K,\lambda\mathbf{x})$, dividing both sides by
 $\|\widehat{\mathbf{x}}_m-\lambda\mathbf{x}\|_2$, the conclusion holding trivially should this norm be zero, using the fact that for any fixed ${\bf t}\in\mathbb{R}^d$, $\expect{y\dotp{\mathbf{a}}{{\bf t}}}=\expect{f_{\mathbf{x}}({\bf t})}$
 gives
 \beas
 \|\widehat{\mathbf{x}}_m-\lambda\mathbf{x}\|_2
 \leq2\sup_{\mathbf t\in \mathcal T}\left|f_{\mathbf{x}}({\bf t})-E[f_{\mathbf{x}}({\bf t})]\right|+2\alpha.
 \enas
 Now applying Lemma \ref{concentration-of-functions} finishes the proof of Theorem \ref{main-theorem}.
 
 \subsection{Preliminaries}
 In addition to chaining, we need the following notions and propositions; we recall the $\psi_q$ norms from Definition \ref{def:psi_qnorm}.
 
 \begin{definition}[Subgaussian random vector]\label{def:vector.subg.norm}
 	A random vector $\mathbf{X}\in\mathbb{R}^d$ is subgaussian if the random variables $\langle\mathbf{X},\mathbf{z}\rangle,\mathbf{z}\in \mathbb{S}^{d-1}$
 	are subgaussian
 	with uniformly bounded subgaussian norm.  The corresponding subgaussian norm of the vector $\mathbf{X}$ is then given by
 	\[\|\mathbf{X}\|_{\psi_2}=\sup_{\mathbf{z}\in\mathbb{S}^{d-1}}\|\langle\mathbf{X},\mathbf{z}\rangle\|_{\psi_2}.\]
 \end{definition}
 
 The proof of the following two propositions are shown in the Appendix.
 \begin{proposition} \label{prop-1}
 	If both $X$ and $Y$ are subgaussian random variables, then $XY$ is an subexponential random variable satisfying
 	\[\|XY\|_{\psi_1}\leq 2\|X\|_{\psi_2}\|Y\|_{\psi_2}.\]
 \end{proposition}
 
 \begin{proposition}\label{prop-3}
 	If $\mathbf{a}$ is a subgaussian random vector with covariance matrix $\mathbf{\Sigma}$, then
 	\[\sigma_{\max}(\mathbf{\Sigma})\leq2\|\mathbf{a}\|_{\psi_2}^2,\]
 	where $\sigma_{\max}(\cdot)$ denotes the maximal singular value of a matrix.
 \end{proposition}
 
In addition, we need the following fact that a vector of $d$ independent subgaussian random variables is subgaussian.
 \begin{proposition}[Lemma 5.24 of \cite{introduction-to-random-matrix}]\label{prop-4}
 	Consider a random vector $\mathbf{a}\in\mathbb{R}^d$, where each entry $a_i$ is an i.i.d. copy of a centered 
 	subgaussian random variable $a$. Then, $\mathbf{a}$ is a subgaussian random vector with norm $\|\mathbf{a}\|_{\psi_2}\leq C\|a\|_{\psi_2}$ where $C$ is a absolute positive constant.
 \end{proposition}

\subsection{Proving Lemma \ref{concentration-of-functions} via Generic Chaining} 
Throughout this section, $C$ denotes an absolute constant whose value may change at each occurrence.
The following notions are necessary ingredients in the generic chaining argument. Let $(\mathcal T,d)$ be a metric space. If $\mathcal{A}_{l}\subseteq \mathcal{A}_{l+1} \subseteq \mathcal T$ for every $l \ge 0$ we say $\{\mathcal{A}_l\}_{l=0}^{\infty}$ is an increasing sequence of subsets of $\mathcal T$. Let $N_0=1$ and $N_l=2^{2^l},~\forall l\geq1$.
\begin{definition}[Admissible sequence]
	An increasing sequence of subsets  $\{\mathcal{A}_l\}_{l=0}^{\infty}$ of $\mathcal T$ is admissible if $|\mathcal{A}_l|\leq N_l$ for all $l \ge 0$.
\end{definition}

Essentially following the framework of Section 2.2 of \cite{Talagrand-book-2}, for each subset $\mathcal{A}_l$, we define $\pi_l:\mathcal T\rightarrow \mathcal{A}_l$ as the closest point map  $\pi_l({\bf t})=\textrm{arg}\min_{\mathbf s\in\mathcal{A}_l}d(\mathbf s,\mathbf t),~\forall \mathbf t\in \mathcal T$. Since each $\mathcal{A}_l$ is a finite set, the minimum is always achievable. If the ${\rm argmin}$ is not unique a representative is chosen arbitrarily. 
The Talagrand $\gamma_2$-functional is defined as
\bea \label{def:TalagrandGamma2}
\gamma_2(\mathcal T,d):=\inf\sup_{\mathbf t\in \mathcal T}\sum_{l=0}^{\infty}2^{l/2}d(\mathbf t,\pi_l(\mathbf t)),\ena
where the infimum is taken with respect to all admissible sequences. 

Though there is no guarantee that $\gamma_2(\mathcal T,d)$ is  finite, the following majorizing measure theorem tells us that its value is comparable to the supremum of a certain Gaussian process.

\begin{lemma}[Theorem 2.4.1 of \cite{Talagrand-book-2}]\label{mmt}
	Consider a family of centered Gaussian random variables $\{G(\mathbf t)\}_{\mathbf t\in\mathcal T}$ indexed by $\mathcal T$, with the canonical distance
	\[d(\mathbf s,\mathbf t)=\expect{(G(\mathbf s)-G(\mathbf t))^2}^{1/2},~\forall \mathbf s,\mathbf t\in \mathcal T.\]
	Then for a universal constant $L$ that does not depend on the covariance of the Gaussian family, we have
	\[\frac1L\gamma_2(\mathcal T,d)\leq\expect{\sup_{\mathbf t\in \mathcal T}G(\mathbf t)}\leq L\gamma_2(\mathcal T,d).\]
\end{lemma}

For $\mathcal T \subseteq \mathbb{R}^d$ and $d(\mathbf{x},\mathbf{y})=\|\mathbf{x}-\mathbf{y}\|_2$ we write $\gamma_2(\mathcal T)$ to denote $\gamma_2(\mathcal T,\|\cdot\|_2)$ defined in \eqref{def:TalagrandGamma2}. Defining the Gaussian process $G({\bf t})=\dotp{\mathbf{g}}{\mathbf t},~\mathbf t\in \mathcal T$, with $\mathbf{g}\sim\mathcal{N}(0,\mathbf{I}_{d\times d})$ we have 
	\[\expect{(G(t)-G(s))^2}^{1/2}=\|t-s\|_2,~\forall t,s\in \mathcal T.\]
When $\mathcal T$ is bounded we may conclude that $\omega(\mathcal T)<\infty$ directly from Definition \ref{def:gmw}, and Lemma \ref{mmt} then implies that Gaussian mean width $\omega(\mathcal T)$ and $\gamma_2(\mathcal T)$ are of the same order, i.e. there exists a universal constant $L \ge 1$ independent of $\mathcal T$ such that
\begin{equation}\label{mmt-gmw}
\frac1L\gamma_2(\mathcal T)\leq\omega(\mathcal T)\leq L \gamma_2(\mathcal T).
\end{equation}

Define
$$\overline{Z}({\bf t})=f_{\mathbf{x}}({\bf t})-\expect{f_{\mathbf{x}}({\bf t})},$$ 
where $f_{\mathbf{x}}({\bf t})$ is as defined in \eqref{def-empirical-function}
and 
$$Z({\bf t})= \frac{1}{m}\sum_{i=1}^m\varepsilon_iy_i\langle\mathbf{a}_i,{\bf t}\rangle,$$
where $\varepsilon_i, i=1,\ldots,m$ are Rademancher variables taking values uniformly in $\{1,-1\}$, independent of each other and of $\{y_i,{\bf a}_i,i=1,2,\ldots,m\}$.

The majority of the proof of Lemma \ref{concentration-of-functions} is devoted to showing that
\begin{equation}\label{symmetry-bound}
P\left[\sup_{{\bf t}\in \mathcal T}\left|Z({\bf t})\right|
\geq C(\|a\|_{\psi_2}^2+\|y\|_{\psi_2}^2)
\frac{\omega(\mathcal T)+u}{\sqrt{m}}\right]
\leq e^{-u} \qmq{for $u \ge 2, m\geq \omega(\mathcal T)^2$,}
\end{equation}
where $C>0$ is a constant. Once \eqref{symmetry-bound} is justified, by the fact $u\geq2$, we have
\begin{equation*}
P\left[\sup_{{\bf t}\in \mathcal T}\left|Z({\bf t})\right|
\geq C(\|a\|_{\psi_2}^2+\|y\|_{\psi_2}^2)
\frac{\omega(\mathcal T)+1}{\sqrt{m}}u\right]
\leq e^{-u} \qmq{for $u \ge 2, m\geq \omega(\mathcal T)^2$.}
\end{equation*}
By Lemma \ref{pr-to-expect}, with $p=1$ and $k=1$, we have
\[\expect{\sup_{ {\bf t} \in \mathcal T}\left|Z({\bf t})\right|}
\leq C(\|a\|_{\psi_2}^2+\|y\|_{\psi_2}^2)\frac{\omega(\mathcal T)+1}{\sqrt{m}}
\]
Thus, invoking
the first bound in the symmetrization lemma,
Lemma \ref{symmetry}, 
\beas
\expect{\sup_{ {\bf t} \in \mathcal T}\left|\overline{Z}({\bf t})\right|}\leq2\expect{\sup_{ {\bf t} \in \mathcal T}\left|Z({\bf t})\right|}
\leq C(\|a\|_{\psi_2}^2+\|y\|_{\psi_2}^2)
\frac{\omega(\mathcal T)+1}{\sqrt{m}}.
\enas
We may then finish the proof of Lemma \ref{concentration-of-functions} using the fact that $u \ge 2$, the second bound in the symmetrization lemma with $\beta=(2C(\|a\|_{\psi_2}^2+\|y\|_{\psi_2}^2)
\omega(\mathcal T)+u)/\sqrt{m}$, and  \eqref{symmetry-bound}, which together imply
\begin{multline*}
P\left[\sup_{ {\bf t} \in \mathcal T}\left|\overline{Z}({\bf t})\right|
\geq C(\|a\|_{\psi_2}^2+\|y\|_{\psi_2}^2)
\frac{\omega(\mathcal T)+u}{\sqrt{m}}\right]\\
\leq 4P\left[\sup_{ {\bf t} \in \mathcal T}\left|Z({\bf t})\right|\geq C(\|a\|_{\psi_2}^2+\|y\|_{\psi_2}^2)
\frac{\omega(\mathcal T)+u}{\sqrt{m}}\right]\leq 4e^{-u}.
\end{multline*}

The rest of the section is devoted to the proof of \eqref{symmetry-bound}.
Pick $\mathbf{t}_0\in \mathcal T$ so that $\{\mathbf{t}_0\}=\mathcal{A}_0\subseteq\mathcal{A}_1\subseteq\mathcal{A}_2\subseteq\mathcal{A}_3\subseteq\cdots$ is an admissible sequence satisfying
\begin{equation}\label{2-appr-gamma}
\sup_{ {\bf t} \in \mathcal T}\sum_{l=0}^\infty2^{l/2}d(t,\pi_l({\bf t}))\leq 2\gamma_2(\mathcal T),
\end{equation}
where we recall $\pi_l$ is the closest point map from $\mathcal T$ to $\mathcal{A}_l$, and the constant 2 on the right hand side of the inequality is introduced to handle the case where the infimum in the definition of $\gamma_2(T)$ is not achieved. 
Then, for any ${\bf t}\in \mathcal T$, we write $Z({\bf t})-Z({\bf t}_0)$ as a telescoping sum, i.e.
\begin{equation}\label{master-telescoping-sum}
Z({\bf t})-Z({\bf t}_0)=\sum_{l=1}^{\infty}Z(\pi_l({\bf t}))-Z(\pi_{l-1}({\bf t}))=\sum_{l=1}^{\infty}\frac1m\sum_{i=1}^m\varepsilon_iy_i\dotp{\mathbf{a}_i}{\pi_l({\bf t})-\pi_{l-1}({\bf t})}.
\end{equation}
Note that this telescoping sum converges with probability 1 because the right hand side of \eqref{2-appr-gamma} is finite.

Then, following ideas in \cite{tail-bound-chaining}, we fix an arbitrary positive integer $p$ and let 
$l_p:=\lfloor\log_2p\rfloor$. Specializing \eqref{master-telescoping-sum} to the case ${\bf t}_0=
\pi_{l_p}({\bf t})$  we obtain, with probability one, that 
\bea \label{master-telescoping-sum.l_p}
Z({\bf t})-Z(\pi_{l_p}({\bf t}))=\sum_{l=l_p+1}^{\infty}Z(\pi_l({\bf t}))-Z(\pi_{l-1}({\bf t}))=\sum_{l=l_p+1}^{\infty}\frac1m\sum_{i=1}^m\varepsilon_iy_i\dotp{\mathbf{a}_i}{\pi_l({\bf t})-\pi_{l-1}({\bf t})}.
\ena

We split the outer index of summation in \eqref{master-telescoping-sum.l_p} into the following two sets
\begin{align*}
I_{1,p}:=\{l>l_p:2^{l/2}\leq\sqrt{m}\} \qmq{and} 
I_{2,p}:=\{l>l_p:2^{l/2}>\sqrt{m}\}.
\end{align*}

On the coarse scale $I_{1,p}$, we have the following lemma:
\begin{lemma}[Coarse scale chaining]\label{coarse}
	For all $p \ge 1$ and $u \geq 2$, there exists a constant $c>0$ such that the inequality
\begin{align*}
\sup_{ {\bf t} \in \mathcal T}\left|\sum_{l\in I_{1,p}}Z(\pi_l({\bf t}))-Z(\pi_{l-1}({\bf t}))\right|\leq 4(\sqrt{2}+1)\|\mathbf{a}\|_{\psi_2}\|y\|_{\psi_2}
\frac{u}{\sqrt{m}}\gamma_2(\mathcal T),
\end{align*}	
holds with probability at least $1-ce^{-pu/4}$.
\end{lemma}
\begin{proof}
	We assume $I_{1,p}$ is non-empty, else the claim is trivial.
	By Proposition \ref{prop-1} and Definition \ref{def:vector.subg.norm}, for any $i\in\{1,2,\cdots,m\}$ we have
	\begin{align*}
	\|\varepsilon_iy_i\langle\mathbf{a}_i,\pi_l({\bf t})-\pi_{l-1}({\bf t})\rangle\|_{\psi_1}
	\leq 2\|\mathbf{a}\|_{\psi_2}\|y\|_{\psi_2}\|\pi_l({\bf t})-\pi_{l-1}({\bf t})\|_2.
	\end{align*}
	Thus, for each $l\in I_{1,p}$, applying Bernstein's inequality (Lemma \ref{Bernstein}) to 
	\beas
	Z(\pi_l({\bf t}))-Z(\pi_{l-1}({\bf t}))=\frac{1}{m}\sum_{i=1}^m\varepsilon_iy_i\dotp{\mathbf{a}_i}{\pi_l({\bf t})-\pi_{l-1}({\bf t})},
	\enas
	an average of independent subexponential random variables, we have that for all $v\geq1$,
	\begin{align*}
	P\left[|Z(\pi_l({\bf t}))-Z(\pi_{l-1}({\bf t}))|\geq 2\|\mathbf{a}\|_{\psi_2}\|y\|_{\psi_2}\left(\frac{\sqrt{2v}}{\sqrt{m}}+\frac{v}{m}\right)\|\pi_l({\bf t})-\pi_{l-1}({\bf t})\|_2\right]\leq 2e^{-v}.
	\end{align*}
	Let $v=2^lu$ for some $u \ge 2$. Using that $2^{l/2}\leq\sqrt{m}$ since $l \in I_{1,p}$, and that $u \ge \sqrt{u}$, we have
	\begin{align}
	P\left[|Z(\pi_l({\bf t}))-Z(\pi_{l-1}({\bf t}))|\geq 2\|\mathbf a\|_{\psi_2}\|y\|_{\psi_2}(\sqrt{2}+1)\frac{u}{\sqrt{m}}2^{l/2}\|\pi_l({\bf t})-\pi_{l-1}({\bf t})\|_2\right]\leq 2\exp(-2^lu). \label{interim-coarse}
	\end{align}
	
	Now for every $l \in I_{1,p}$ and ${\bf t} \in {\mathcal T}$ define the event
	\begin{align*}
	\Omega_{l,{\bf t}}=\left\{\omega: |Z(\pi_l({\bf t}))-Z(\pi_{l-1}({\bf t}))|\geq 2(\sqrt{2}+1)\|\mathbf a\|_{\psi_2}\|y\|_{\psi_2}\frac{u}{\sqrt{m}}2^{l/2}\|\pi_l({\bf t})-\pi_{l-1}({\bf t})\|_2\right\},
	\end{align*} 
	and let $\Omega:=\bigcup_{l\in I_{1,p}}\bigcup_{ {\bf t} \in\mathcal T}\Omega_{l,{\bf t}}$. As $\mathcal{A}_{l}=\{\pi_{l}({\bf t})\}_{ {\bf t} \in \mathcal T}$ contains 
	at most $2^{2^l}$ points, it follows that the union over ${\bf t} \in \mathcal{T}$ in the definition of $\Omega$ can be written as a union over at most $2^{2^{l+1}}$ indices.
Hence, with $u \ge 2$, Lemma \ref{union-bound} with $k=1$ may now be invoked to yield
	\begin{align*}
	P\left[\bigcup_{l\in I_{1,p},{\bf t}\in \mathcal T}\Omega_{l,{\bf t}}\right]\leq ce^{-pu/4},
	\end{align*}
	for some $c>0$. Thus, on the event $\Omega^c$, we have
	\begin{align*}
	\sup_{ {\bf t} \in \mathcal T}\left|\sum_{l\in I_{1,p}}Z(\pi_l({\bf t}))-Z(\pi_{l-1}({\bf t}))\right|
	\leq&\sup_{ {\bf t} \in \mathcal T}\sum_{l\in I_{1,p}}\left|Z(\pi_l({\bf t}))-Z(\pi_{l-1}({\bf t}))\right|\\
	\leq&  \sup_{ {\bf t} \in \mathcal T}2(\sqrt{2}+1)\|\mathbf a\|_{\psi_2}\|y\|_{\psi_2}
	\frac{u}{\sqrt{m}}
	\sum_{l\in I_1}2^{l/2}\|\pi_l({\bf t})-\pi_{l-1}({\bf t})\|_2\\
	\leq&\sup_{ {\bf t} \in \mathcal T}2(\sqrt{2}+1)\|\mathbf a\|_{\psi_2}\|y\|_{\psi_2}
	\frac{u}{\sqrt{m}}
	\sum_{l=1}^{\infty}2^{l/2}\|\pi_l({\bf t})-\pi_{l-1}({\bf t})\|_2\\
	\leq&4(\sqrt{2}+1)\|\mathbf a\|_{\psi_2}\|y\|_{\psi_2}\frac{u}{\sqrt{m}}\gamma_2(\mathcal T),
	\end{align*}
	where the last inequality follows from \eqref{2-appr-gamma}, finishing the proof.
\end{proof}


For the finer scale chaining, we will apply the following lemma whose proof is in the appendix.
\begin{lemma}\label{subg-tail}
	For any $\mathbf{t}\in \mathbb{R}^d$, $u\geq1$ and $2^{l/2}>\sqrt{m}$, we have
	\begin{align*}
	P\left[\left(\frac1m\sum_{i=1}^m\langle\mathbf{a}_i,\mathbf{t}\rangle^2\right)^{1/2}
	\geq \sqrt{5+3\sqrt2}
	\|\mathbf{a}\|_{\psi_2}
	\sqrt{\frac{u}{m}}
	2^{l/2}
	\|t\|_2\right]
	\leq 2\exp(-2^lu).
	\end{align*}
\end{lemma}

\begin{lemma}[Finer scale chaining]\label{finer-chaining}
	Let
	$$Y_m=\left|\frac1m\sum_{i=1}^my_i^2-\expect{y^2}\right|.$$
	Then for all $p \ge 1$, with probability at least $1-ce^{-pu/4}$
	\begin{align*}
	\sup_{ {\bf t} \in \mathcal T}\left|\sum_{l\in I_{2,p}}Z(\pi_l({\bf t}))-Z(\pi_{l-1}({\bf t}))\right|
	\leq2\sqrt{5+3\sqrt2}\left(Y_m^{1/2}+\sqrt{2}\|y\|_{\psi_2}\right)\|\mathbf{a}\|_{\psi_2}
	\sqrt{\frac{u}{m}}
	\gamma_2(\mathcal T),
	\end{align*} 
	with some constant $c>0$ and $u\geq2$.
\end{lemma}

\begin{proof}
	For any $p \ge 1, l\in I_{2,p}$ and $t\in \mathcal T$, by the Cauchy-Schwarz inequality,
	\begin{multline*}
	|Z(\pi_l({\bf t}))-Z(\pi_{l-1}({\bf t}))|=\left|\frac1m\sum_{i=1}^m\varepsilon_iy_i\langle\mathbf{a}_i,\pi_l({\bf t})-\pi_{l-1}({\bf t})\rangle\right|\\
	\leq\left(\frac1m\sum_{i=1}^my_i^2\right)^{1/2}
	\cdot\left(\frac1m\sum_{i=1}^m\left|\langle\mathbf{a}_i,\pi_l({\bf t})-\pi_{l-1}({\bf t})\rangle\right|^2\right)^{1/2}
	\end{multline*}
	Since $y$ is subgaussian, $\expect{y^2}\leq2\|y\|_{\psi_2}^2$. Thus,
	\[\left(\frac1m\sum_{i=1}^my_i^2\right)^{1/2}=\left(\frac1m\sum_{i=1}^my_i^2-\expect{y^2}+\expect{y^2}\right)^{1/2}
	\leq Y_m^{1/2}+\sqrt{2}\|y\|_{\psi_2}.\]
	Furthermore, by Lemma \ref{subg-tail}, for any $l\in I_{2,p}$, we have
	\[P\left[\left(\frac1m\sum_{i=1}^m\left|\langle\mathbf{a}_i,\pi_l({\bf t})-\pi_{l-1}({\bf t})\rangle\right|^2\right)^{1/2}
	\geq \sqrt{5+3\sqrt2}\|\mathbf{a}\|_{\psi_2}
	\sqrt{\frac{u}{m}}
	2^{l/2} \|\pi_l({\bf t})-\pi_{l-1}({\bf t})\|_2\right]
	\leq 2\exp(-2^lu).\]
	Thus, combining the above two inequalities,
	\begin{multline*}
	P\left[|Z(\pi_l({\bf t}))-Z(\pi_{l-1}({\bf t}))|\geq\sqrt{5+3\sqrt2}\left(Y_m^{1/2}+\sqrt{2}\|y\|_{\psi_2}\right)\|\mathbf{a}\|_{\psi_2}
	\sqrt{\frac{u}{m}}
	2^{l/2}\|\pi_l({\bf t})-\pi_{l-1}({\bf t})\|_2\right]\\
	\leq  2\exp(-2^lu).
	\end{multline*}
	The rest of the proof follows a standard chaining argument similar to the proof of Lemma \ref{coarse} after \eqref{interim-coarse} and is not repeated here for brevity.
\end{proof}

Now we are ready to prove Lemma \ref{concentration-of-functions}, for which we have already demonstrated the sufficiency of \eqref{symmetry-bound}.

\begin{proof}[Proof of \eqref{symmetry-bound}]
	First, for all $p \ge 1$ and $u \ge 2$, by Lemma \ref{finer-chaining}, with probability at least $1-ce^{-pu/4}$, 
	\begin{align*}
	&\sup_{ {\bf t} \in \mathcal T}\left|\sum_{l\in I_{2,p}}Z(\pi_l({\bf t}))-Z(\pi_{l-1}({\bf t}))\right|\\
	\leq&2\sqrt{5+3\sqrt2}Y_m^{1/2}\|\mathbf{a}\|_{\psi_2}
	\sqrt{\frac{u}{m}}
	\gamma_2(\mathcal T)
	+2\sqrt{8+6\sqrt2}\|\mathbf{a}\|_{\psi_2}\|y\|_{\psi_2}
	\sqrt{\frac{u}{m}}
	\gamma_2(\mathcal T)\\
	\leq&Y_m+(5+3\sqrt2)\|\mathbf a\|_{\psi_2}^2\frac{u}{m}\gamma_2(\mathcal T)^2
	+2\sqrt{8+6\sqrt2}\|\mathbf{a}\|_{\psi_2}\|y\|_{\psi_2}
	\sqrt{\frac{u}{m}}
	\gamma_2(\mathcal T),
	\end{align*}
	where we applied the inequality $2ab\leq a^2+b^2$ on the first term. Then, combining with Lemma \ref{coarse}, we have with probability at least $1-ce^{-pu/4}$,
	\begin{multline*}
	\sup_{ {\bf t} \in \mathcal T}\left|Z({\bf t})-Z(\pi_{l_p}({\bf t}))\right|\\
	\leq \sup_{ {\bf t} \in \mathcal T}\left|\sum_{l\in I_{1,p}}Z(\pi_l({\bf t}))-Z(\pi_{l-1}({\bf t}))\right|
	+\sup_{ {\bf t} \in \mathcal T}\left|\sum_{l\in I_{2,p}}Z(\pi_l({\bf t}))-Z(\pi_{l-1}({\bf t}))\right|\\
	\leq Y_m+(5+3\sqrt2)\|\mathbf a\|_{\psi_2}^2\frac{u}{m}\gamma_2(\mathcal T)^2
	+2\sqrt{8+6\sqrt{2}}\|\mathbf a\|_{\psi_2}\|y\|_{\psi_2}
	\sqrt{\frac{u}{m}}
	\gamma_2(\mathcal T)\\
	+4(\sqrt{2}+1)\|\mathbf a\|_{\psi_2}\|y\|_{\psi_2}
	\frac{u}{\sqrt{m}}\gamma_2(\mathcal T)\\
	\le Y_m+(5+3\sqrt2)\|\mathbf a\|_{\psi_2}^2\frac{u}{m}\gamma_2(\mathcal T)^2
		+\left( \sqrt{8+6\sqrt{2}}+ 2(\sqrt{2}+1) \right) 2\|\mathbf a\|_{\psi_2}\|y\|_{\psi_2}\frac{u}{\sqrt{m}}
		\gamma_2(\mathcal T).
	\end{multline*}

	By the conditions in \eqref{symmetry-bound} we have 
	$m\geq \omega(\mathcal T)^2$. Using inequality \eqref{mmt-gmw} on the relation between $\omega(\mathcal T)$ and $\gamma_2(\mathcal T)$ gives $m\geq \gamma_2(\mathcal T)^2/L^2$. Thus, $\gamma_2(\mathcal T)^2/m\leq L\gamma_2(\mathcal T)/\sqrt{m}$, and the second term is bounded by
	\beas
	(5+3\sqrt2)L \|\mathbf a\|_{\psi_2}^2\frac{u}{\sqrt{m}}\gamma_2(\mathcal T).
	\enas

	For the last term we apply the bound $2\|\mathbf a\|_{\psi_2}\|y\|_{\psi_2}\leq\|\mathbf a\|_{\psi_2}^2+\|y\|_{\psi_2}^2$
	Thus, with probability at least $1-ce^{-pu/4}$,
	\begin{equation}
	\sup_{ {\bf t} \in \mathcal T}\left|Z({\bf t})-Z(\pi_{l_p}({\bf t}))\right|\leq Y_m + C\left(\|\mathbf a\|_{\psi_2}^2+\|y\|_{\psi_2}^2\right)\frac{u\gamma_2(\mathcal T)}{\sqrt{m}}, \nonumber
	\end{equation}
	for the constant 
	\[C=5L+2+(3L+2)\sqrt{2}+\sqrt{8+6\sqrt{2}}.\]
	By Proposition \ref{prop-4}, $\|\mathbf a\|_{\psi_2}\leq C\|a\|_{\psi_2}$ for some constant $C$. Thus, with probability at least $1-ce^{-pu/4}$, for some constant $C$ large enough,
	\beas 
	\sup_{ {\bf t} \in \mathcal T}\left|Z({\bf t})-Z(\pi_{l_p}({\bf t}))\right|\leq Y_m + C\left(\|a\|_{\psi_2}^2+\|y\|_{\psi_2}^2\right)\frac{u\gamma_2(\mathcal T)}{\sqrt{m}},
	\enas
    or, equivalently
	\beas
	\xi \le C\left(\|a\|_{\psi_2}^2+\|y\|_{\psi_2}^2\right)\frac{u\gamma_2(T)}{\sqrt{m}} \qmq{where} \xi = \max\left\{\sup_{ {\bf t} \in \mathcal T}\left|Z({\bf t})-Z(\pi_{l_p}({\bf t}))\right|-Y_m,0\right\}.
	\enas
	Invoking Lemma \ref{pr-to-expect} with $k=1$, for all $1 \le p < \infty$
	\[\expect{\xi^p}^{1/p}\leq C\left(\|a\|_{\psi_2}^2+\|y\|_{\psi_2}^2\right)\frac{\gamma_2(\mathcal T)}{\sqrt{m}}.\]
	Since
	\begin{multline*} 
	\xi\geq\max\left\{\sup_{ {\bf t} \in \mathcal T}\left|Z({\bf t})-Z(\pi_{l_p}({\bf t}))\right|,0\right\}-Y_m=\sup_{ {\bf t} \in \mathcal T}\left|Z({\bf t})-Z(\pi_{l_p}({\bf t}))\right|-Y_m\\
	\geq\sup_{ {\bf t} \in \mathcal T}\left|Z({\bf t})\right|-\sup_{ {\bf t} \in \mathcal T}\left|Z(\pi_{l_p}({\bf t}))\right|-Y_m,
	\end{multline*}
	and $\xi$ and $Y_m$ are both non-negative, by Minkowski's inequality it follows that
	\begin{multline} \label{inter-three-parts}
	\expect{\left(\sup_{ {\bf t} \in \mathcal T}\left|Z({\bf t})\right|\right)^p}^{1/p}\leq\expect{\left(\xi+\sup_{ {\bf t} \in \mathcal T}\left|Z(\pi_{l_p}({\bf t}))\right|+Y_m\right)^p}^{1/p}
	\\
	\leq\expect{\xi^p}^{1/p}+\expect{\left(\sup_{ {\bf t} \in \mathcal T}\left|Z(\pi_{l_p}({\bf t}))\right|\right)^p}^{1/p}
	+\expect{Y_m^p}^{1/p}\\
	\leq C\left(\|a\|_{\psi_2}^2+\|y\|_{\psi_2}^2\right)\frac{\gamma_2(\mathcal T)}{\sqrt{m}}
	+\expect{\left(\sup_{ {\bf t} \in \mathcal T}\left|Z(\pi_{l_p}({\bf t}))\right|\right)^p}^{1/p}
	+\expect{Y_m^p}^{1/p}.
	\end{multline}

	For the second term, we have 
	\begin{align*}
	\expect{\left(\sup_{ {\bf t} \in \mathcal T}\left|Z(\pi_{l_p}({\bf t}))\right|\right)^p}
	\leq\sum_{\mathbf{t}\in\mathcal{A}_{l_p}}\expect{|Z({\bf t})|^p}
	\leq|\mathcal{A}_{l_p}|\sup_{ {\bf t} \in \mathcal T}\expect{|Z({\bf t})|^p}
	\leq 2^p\sup_{ {\bf t} \in \mathcal T}\expect{|Z({\bf t})|^p},
	\end{align*}
	where the first inequality follows from the fact that $\pi_{l_p}(\cdot)$ can only take values in ${\cal A}_{l_p}$, and the last inequality follows from the fact that $l_p=\lfloor\log_2p\rfloor$. On the other hand, applying
		Proposition \ref{prop-4}, yielding that $\|{\bf a}\|_{\psi_2} \le C\|a\|_{\psi_2}$, and Proposition \ref{prop-1},
	by a direct application of Bernstein's inequality (Lemma \ref{Bernstein}) we have, for any fixed $\mathbf{t}\in \mathcal T$, 
	\[P\left[|Z({\bf t})|\geq 2C\|y\|_{\psi_2}\|a\|_{\psi_2}(1+\sqrt2)\frac{pu}{\sqrt{m}}\right]\leq2e^{-pu}, \qmq{whenever $pu\geq0$.}\]
	Hence, applying Lemma \ref{pr-to-expect} with $k=1$, for all $1 \le p < \infty$, 
	\[\expect{|Z({\bf t})|^p}^{1/p}\leq \frac{C\|y\|_{\psi_2}\|a\|_{\psi_2}p}{\sqrt{m}},\]
	for all $t\in \mathcal T$ and some constant $C>0$. Thus, 
	\begin{equation}\label{inter-one-part}
	\expect{\left(\sup_{ {\bf t} \in \mathcal T}\left|Z(\pi_{l_p}({\bf t}))\right|\right)^p}^{1/p}\leq\frac{2C\|y\|_{\psi_2}\|a\|_{\psi_2}p}{\sqrt{m}}\leq\frac{C\left(\|y\|_{\psi_2}^2+\|a\|_{\psi_2}^2\right)p}{\sqrt{m}}.
	\end{equation}

Now consider $\expect{Y_m^p}^{1/p}$, the final term in \eqref{inter-three-parts}, recalling that  $Y_m=\frac1m\sum_{i=1}^m(y_i^2-\expect{y_i^2})$. Applying Proposition \ref{prop-1}, we have
		$$\|y_i^2-\expect{y_i^2}\|_{\psi_1}\leq
		\|y_i^2\|_{\psi_{1}}+\expect{y_i^2}
		\leq2\|y_i\|_{\psi_2}^2+2\|y_i\|_{\psi_2}^2=4\|y\|_{\psi_2}^2.$$
		Thus, using Bernstein's inequality and Lemma \ref{pr-to-expect} as before, we obtain
		\[Pr\left[Y_m\geq4(1+\sqrt{2})\|y\|_{\psi_2}^2\frac{pu}{\sqrt{m}}\right]\leq2e^{-pu},~\forall pu\geq0.\]
	and
	\begin{equation}\label{inter-two-part}
	\expect{Y_m^p}^{1/p}\leq\frac{C\|y\|_{\psi_2}^2p}{\sqrt{m}}\leq\frac{C\left(\|y\|_{\psi_2}^2+\|a\|_{\psi_2}^2\right)p}{\sqrt{m}}.
	\end{equation}
	Combining \eqref{inter-three-parts}, \eqref{inter-one-part} and \eqref{inter-two-part} gives
	\[\expect{\left(\sup_{ {\bf t} \in \mathcal T}\left|Z({\bf t})\right|\right)^p}^{1/p}\leq\frac{C\left(\|y\|_{\psi_2}^2+\|a\|_{\psi_2}^2\right)(\gamma_2(\mathcal T)+p)}{\sqrt{m}},\]
	for some constant $C>0$. Since this inequality holds for any $p\geq1$, applying Lemma \ref{exp-to-pr} with $k=1$ yields
	\[P\left[\sup_{ {\bf t} \in \mathcal T}\left|Z({\bf t})\right|
	\geq C(\|y\|_{\psi_2}^2+\|a\|_{\psi_2}^2)\frac{\gamma_2(\mathcal T)+u}{\sqrt{m}}\right]
	\leq e^{-u}.\]
	 The proof of \eqref{symmetry-bound} is now completed by invoking Lemma \ref{mmt}, which gives $\gamma_2(\mathcal T)\leq L\omega(\mathcal T)$ for some constant $L \ge 1$.
\end{proof}

\appendix
\section{Additional lemmas} \label{app:A}
The following lemma is one version of the contraction principle;
for a proof see \cite{Talagrand-book}:

\begin{lemma}\label{contraction}
	Let $F:[0,\infty) \rightarrow [0,\infty)$ be convex and nondecreasing. Let $\{\eta_i\}$ and $\{\xi_i\}$ be two symmetric sequences of real valued random variables such that for some constant $C \geq1$ for every $i$ and $t>0$ we have
	\[P[|\eta_i|>t]\leq C \cdot P[|\xi_i|>t].\]
	Then, for any finite sequence $\{\mathbf{x}_i\}$ in a vector space with semi-norm $\|\cdot\|$,
	\[\expect{F\left(\left\|\sum_i\eta_i\mathbf{x}_i\right\|\right)}
	\leq\expect{F\left(C \cdot\left\|\sum_i\xi_i\mathbf{x}_i\right\|\right)}.\]
\end{lemma}
\begin{remark}
Though Lemma 4.6 of \cite{Talagrand-book} states the contraction principle in a Banach space, the proofs of Theorem 4.4 and Lemma 4.6 of \cite{Talagrand-book} hold for vector spaces under any semi-norm.
\end{remark}

The following symmetrization lemma is the same as Lemma 4.6 of \cite{ALPV14}.
\begin{lemma}[Symmetrization]\label{symmetry}
	Let
	\beas
	\overline{Z}({\bf t})=f_{\mathbf{x}}({\bf t})-\expect{f_{\mathbf{x}}({\bf t})} \qmq{where} f_{{\bf x}}({\bf t})=\frac{1}{m}\sum_{i=1}^m y_i \langle {\bf a}_i,\mathbf t \rangle,
	\enas
	and 
	$$Z({\bf t})= \frac{1}{m}\sum_{i=1}^m\varepsilon_iy_i\langle\mathbf{a}_i,\mathbf t\rangle,$$
	where $\{ \varepsilon_i: 1 \le i \le m\}$ is a collection of Rademacher random variables, each uniformly distributed over $\{-1,1\}$, and independent of each other and of $\{y_i,{\bf a}_i: 1 \le i \le m\}$.
	Then for any measurable set $\mathcal T\subset \mathbb{R}^d$,
	$$\expect{\sup_{ {\bf t} \in \mathcal T}|\overline{Z}({\bf t})|}
	\leq 2\expect{\sup_{ {\bf t} \in \mathcal T}|Z({\bf t})|},$$
	and for any $\beta>0$
	\[
	P\left[\sup_{ {\bf t} \in \mathcal T}|\overline{Z}({\bf t})|\geq 2\expect{\sup_{ {\bf t} \in \mathcal T}|\overline{Z}({\bf t})|}+\beta\right]\leq 4P\left[\sup_{ {\bf t} \in \mathcal T}|Z({\bf t})|\geq \beta/2\right].\]
\end{lemma}

\begin{lemma}[Lemma A.4 of \cite{tail-bound-chaining}]\label{union-bound}
	Fix $1\leq p <\infty$, $0<k<\infty$, $u\geq2$ and $l_p:=\lfloor\log_2p\rfloor$. For every $l>l_p$, let
	$J_l$ be an index set such that $|J_l|\leq2^{2^{l+1}}$,
	and
	 $\left\{\Omega_{l,i}\right\}_{i\in J_l}$
	a collection of events satisfying
	\[P\left[\Omega_{l,i}\right]\leq 2\exp(-2^lu^k),~\forall i\in J_l.\]
	Then there exists an
	absolute constant $c\leq16$ such that
	\[P\left[\cup_{l>l_p}\cup_{i\in J_l}\Omega_{l,i}\right]\leq c\exp(-pu^k/4).\]
\end{lemma}

\begin{lemma}[Lemma A.5 of \cite{tail-bound-chaining}]\label{pr-to-expect}
	Fix $1\leq p <\infty$ and $0<k<\infty$. Let $\beta\geq0$
	and suppose that $\xi$ is a nonnegative random variable such that for some $c,u_*>0$,
	\[P\left[\xi>\beta u\right]\leq c\exp(-pu^k/4),\quad \forall u\geq u_*.\]
	Then for a constant $\tilde{c}_k>0$ depending only on $k$,
	\[\expect{\xi^p}^{1/p}\leq\beta(\tilde{c}_kc+u_*).\]
\end{lemma}

\begin{lemma}[Proposition 7.11 of  \cite{FR13}]\label{exp-to-pr}
	If $X$ is a non-negative random variable satisfying 
	\[\expect{X^p}^{1/p}\leq b+ ap^{1/k} \quad \forall p\geq1,\]
	for positive real numbers $a$ and $k$, and $b\geq 0$, then, for any $u\geq1$,
	\[P\left[X\geq e^{1/k}(b+au)\right]\leq\exp(-u^k/k).\]
\end{lemma}

Finally, for the following result see Theorem 2.10 of \cite{BLM13}.
\begin{lemma}[Bernstein's inequality]\label{Bernstein}
	Let $X_1,\cdots,X_m$ be a sequence of independent, mean zero random variables.
	 If there exist positive constants $\sigma$ and $D$ such that
	\[\frac1m\sum_{i=1}^m\expect{|X_i|^p}\leq\frac{p!}{2}\sigma^2D^{p-2},~p=2,3,\cdots\]
	then for any $u\geq0$,
	\[P\left[\left|\frac1m\sum_{i=1}^mX_i\right|\geq\frac{\sigma}{\sqrt{m}}\sqrt{2u}+\frac{D}{m}u\right]
	\leq2\exp(-u).\]
	If $X_1,\cdots,X_m$ are all subexponential random variables, then, $\sigma$ and $D$ can be chosen as $\sigma=\frac{1}{m}\sum_{i=1}^m\|X_i\|_{\psi_1}$ and $D=\max_i\|X_i\|_{\psi_1}$. 
\end{lemma}

\section{Additional proofs}\label{app:additional} 

With $g$ a standard normal variable, we begin by considering the solution $f$ to \eqref{stein's-equation}, the special case of the Stein equation
\bea \label{eq:stein.equation}
f'(x)-xf(x)=h(x)-Eh(g),
\ena
with the specific choice of test function $h(x)=|x|$.

\begin{lemma}  \label{lem:stein.sol.abs}
	The solution $f$ of \eqref{stein's-equation} satisfies $\|f''\|=1$.
\end{lemma}
\begin{proof}
	In general, when $f$ solves \eqref{eq:stein.equation} for a given test function $h(\cdot)$ then $-f(-x)$ solves \eqref{eq:stein.equation} for $h(-\cdot)$. As in the case at hand $h(x)=|x|$, for which $h(-x)=h(x)$, it suffices to show that $0 \le f(x) \le 1$ for all $x >0$, over which range 
	\eqref{stein's-equation} specializes to
	\begin{equation}\label{stein-positive}
	f'(x)-xf(x)=x-\sqrt{\frac2\pi}.
	\end{equation}
	Taking derivative on both sides yields
	\[f''(x)-f(x)-xf'(x)=1,\]
	and combining the above two equalities gives
	\bea \label{eq:fdouble}
	f''(x)=(1+x^2)f(x)+x\left(x-\sqrt{\frac2\pi}\right)+1.
	\ena
	On the other hand, solving \eqref{stein-positive} via integrating factors gives, for all $x > 0$,
	\begin{multline} \label{eq:explicit.Sol.Abs}
	f(x)=-e^{x^2/2}\int_x^\infty\left(z-\sqrt{\frac2\pi}\right)e^{-z^2/2}dz
	=-1+2e^{x^2/2}\int_x^\infty\frac{e^{-z^2/2}}{\sqrt{2\pi}}dz=\\-1+2e^{x^2/2}(1-\Phi(x)),
	\end{multline}
	where $\Phi(\cdot)$ is the cumulative distribution function of the standard normal.

	For any $x>0$, by classical upper and lower tail bounds for $\Phi(\cdot)$, we have 
		$$
		\frac{x}{\sqrt{2 \pi}(1+x^2)}
		\leq
		e^{x^2/2}(1-\Phi(x))\leq\min\left\{\frac12,\frac{1}{x\sqrt{2\pi}}\right\},$$
		which in turn implies, using \eqref{eq:fdouble} and \eqref{eq:explicit.Sol.Abs}, that for all $x>0$
		\beas
		0 \le f''(x) \le \min\left\{x\left(x-\sqrt{\frac{2}{\pi}}\right)+1,\frac{1}{x}\sqrt{\frac{2}{\pi}}\right\}.
		\enas
		Handling the cases $0<x\le \sqrt{2/\pi}$ and $x>\sqrt{2/\pi}$ separately, we see $0 \le f''(x) \le 1$ for all $x>0$, as desired.
\end{proof}

\begin{proof}[Proof of Proposition \ref{prop-1}]
	We may assume $\|Y\|_{\psi_1}\not =0$ as the inequality is trivial otherwise. By definition $\|XY\|_{\psi_1}=\sup_{p\geq1}p^{-1}\expect{|XY|^p}^{1/p}$. Applying  $2ab\leq a^2+b^2$ and Minkowski's inequality, for any $\epsilon>0$,
	\begin{align*}
	\expect{|XY|^p}^{1/p}\leq\expect{\left|\frac{X^2}{2\epsilon}+\frac{\epsilon Y^2}{2}\right|^p}^{1/p}
	\leq\frac{1}{2\epsilon}\expect{X^{2p}}^{1/p}+\frac\epsilon2\expect{Y^{2p}}^{1/p}.
	\end{align*}
	Applying the definition of the $\psi_1$ norm, this inequality implies
	\[\|XY\|_{\psi_1}\leq\frac{1}{2\epsilon}\|X^2\|_{\psi_1}+\frac{\epsilon}{2}\|Y^2\|_{\psi_1}.\]
	The term $\|X^2\|_{\psi_1}$ can be bounded as follows,
	\begin{align*}
	\|X^2\|_{\psi_1}=&\sup_{p\geq1}\left(p^{-1/2}\expect{X^{2p}}^{1/2p}\right)^2
	=2 \sup_{p\geq1}\left((2p)^{-1/2}\expect{X^{2p}}^{1/2p}\right)^2
	\leq2\|X\|_{\psi_2}^2.
	\end{align*}
	Arguing similarly for $Y$,
	\[\|XY\|_{\psi_1}\leq\frac{1}{\epsilon}\|X\|^2_{\psi_2}+\epsilon\|Y\|^2_{\psi_2},\]
	and choosing $\epsilon=\|X\|_{\psi_2}/\|Y\|_{\psi_2}$ finishes the proof.
\end{proof}

\begin{proof}[Proof of Proposition \ref{prop-3}]
	By definition, we have
	\begin{align*}
	\|\mathbf{a}\|_{\psi_2}=&\sup_{\mathbf{z}\in\mathbb{S}^{d-1}}\|\langle\mathbf{a},\mathbf{z}\rangle\|_{\psi_2}\\
	=&\sup_{\mathbf{z}\in\mathbb{S}^{d-1}}
	\sup_{p\geq1}\frac{1}{p^{1/2}}\expect{|\langle\mathbf{a},\mathbf{z}\rangle|^p}^{1/p}\\
	\geq&\sup_{\mathbf{z}\in\mathbb{S}^{d-1}}\frac{1}{\sqrt{2}}\expect{\langle\mathbf{a},\mathbf{z}\rangle^2}^{1/2}\\
	=&\frac{1}{\sqrt{2}}\sup_{\mathbf{z}\in\mathbb{S}^{d-1}}\langle\mathbf{\Sigma}\mathbf{z},\mathbf{z}\rangle^{1/2}\\
	=&\frac{1}{\sqrt{2}}\sigma_{\max}(\mathbf{\Sigma})^{1/2},
	\end{align*}
	and squaring both sides finishes the proof.
\end{proof}

%

\begin{proof}[Proof of Lemma \ref{subg-tail}]
	Since $\langle\mathbf{a}_i,\mathbf{t}\rangle$ is subgaussian, it follows, $\langle\mathbf{a}_i,\mathbf{t}\rangle^2$ is subexponential by Proposition \ref{prop-1}. Note that $\expect{\langle\mathbf{a}_i,\mathbf{t}\rangle^2}\leq \sigma_{\max}(\mathbf\Sigma)\|\mathbf{t}\|_2^2\leq2\|\mathbf{a}\|_{\psi_2}^2\|\mathbf{t}\|_2^2$ by Proposition \ref{prop-3}. Then, by Remark \ref{norm-justify} and Proposition \ref{prop-1}
	\[\left\|\langle\mathbf{a}_i,\mathbf{t}\rangle^2-\expect{\langle\mathbf{a}_i,\mathbf{t}\rangle^2}\right\|_{\psi_1}
	\leq\left\|\langle\mathbf{a}_i,\mathbf{t}\rangle^2\right\|_{\psi_1}+2\|\mathbf{a}\|_{\psi_2}^2\|\mathbf{t}\|_2^2
	\leq3\|\mathbf a\|_{\psi_2}^2\|\mathbf{t}\|_2^2.\]
	Now an application of Bernstein's inequality (Lemma \ref{Bernstein}) gives,

	\[P\left[\left(\frac1m\sum_{i=1}^m\langle\mathbf{a}_i,\mathbf{t}\rangle^2-\expect{\langle\mathbf{a}_i,\mathbf{t}\rangle^2}\right)\geq 
	3\|\mathbf a\|_{\psi_2}^2\left(\frac{\sqrt{2v}}{\sqrt{m}}+\frac{v}{m}\right)\|\mathbf{t}\|_2^2\right]
	\leq 2e^{-v}.\]
	We let $v=2^{l}u$ and apply the hypothesis $2^{l/2}>\sqrt{m}$ and $u\geq1$ to obtain
	\[P\left[\left(\frac1m\sum_{i=1}^m\langle\mathbf{a}_i,\mathbf{t}\rangle^2-\expect{\langle\mathbf{a}_i,\mathbf{t}\rangle^2}\right)
	\geq 3\left(1+\sqrt{2}\right)
	\|\mathbf a\|_{\psi_2}^2
	\frac{2^lu}{m}\|\mathbf{t}\|_2^2\right]
	\leq 2\exp{(-2^lu)}.\]
	Thus, by $2^{l/2}>\sqrt{m}$ and $u\geq1$ again,
	\[P\left[\left(\frac1m\sum_{i=1}^m\langle\mathbf{a}_i,\mathbf{t}\rangle^2\right)
	\geq \left(3
	\left(1+\sqrt{2}\right)+2\right)\frac{2^lu}{m}
	\|\mathbf a\|_{\psi_2}^2\|\mathbf{t}\|_2^2\right]
	\leq 2\exp{(-2^lu)},\]
	which yields the claim upon taking square roots on both sides of the first inequality.
\end{proof}

\bibliographystyle{IMAIAI}
\bibliography{stein-chaining-arXiv}

\end{document}